\documentclass[pdfmx, a4paper, 11pt]{article}

\title{Moduli spaces of logarithmic connections associated with difference Painlev\'e equations of type $E^{(1)}_n$}
\author{Takafumi Matsumoto\footnote{National Institute of Technology, Kitakyushu College, 5-20-1 Shii, Kokuraminamiku, Kitakyushu, Fukuoka, 802-0985 Japan, Email address: tfmatumt@gmail.com}}
\date{}

\usepackage[top=30truemm, bottom=30truemm, left=30truemm, right=30truemm]{geometry}
\usepackage{amsmath,amsthm, mathrsfs}
\usepackage{amssymb}
\usepackage{latexsym}
\usepackage{tikz}
\usepackage{tikz-cd}
\usepackage{here}
\usepackage{bm}
\usepackage[all]{xy}

\theoremstyle{definition}
\newtheorem{definition}{Definition}[section]
\newtheorem{theorem}[definition]{Theorem}
\newtheorem{proposition}[definition]{Proposition}
\newtheorem{assumption}[definition]{Assumption}
\newtheorem{lemma}[definition]{Lemma}
\newtheorem{corollary}[definition]{Corollary}

\newtheorem{remark}[definition]{Remark}
              
\newcommand{\mbC}{{\mathbb C}}        
          
\newcommand{\mbL}{\mathbb{L}}   
\newcommand{\mbP}{\mathbb{P}}
\newcommand{\mbR}{\mathbb{R}}
\newcommand{\mbZ}{\mathbb{Z}}

\newcommand{\mcD}{\mathcal{D}}
\newcommand{\mcE}{\mathcal{E}}
\newcommand{\mcF}{\mathcal{F}}
\newcommand{\mcG}{\mathcal{G}}

\newcommand{\mcI}{\mathcal{I}}
\newcommand{\mcJ}{\mathcal{J}}

\newcommand{\mcM}{\mathcal{M}}
\newcommand{\mcN}{\mathcal{N}}
\newcommand{\mcO}{\mathcal{O}}

\newcommand{\mcS}{\mathcal{S}}

\newcommand{\mcX}{\mathcal{X}}
\newcommand{\mcY}{\mathcal{Y}}

\newcommand{\Aut}{\textrm{Aut}}
\newcommand{\Coker}{\textrm{Coker}\,}

\newcommand{\res}{\textrm{res}}
\newcommand{\End}{\textrm{End}}
\newcommand{\rank}{\textrm{rank}\,}
\newcommand{\App}{\textrm{App}}
\newcommand{\Image}{\textrm{Im}\,}
\newcommand{\Ker}{\textrm{Ker}\,}
\newcommand{\tr}{\textrm{tr}}
\newcommand{\id}{\textrm{id}}
\newcommand{\Hom}{\textrm{Hom}}
\newcommand{\length}{\textrm{length}\,}

\newcommand{\Spec}{\textrm{Spec}\,}
\newcommand{\Sym}{\textrm{Sym}\,}

\newcommand{\Pic}{\textrm{Pic}\,}

\newcommand{\pl}{\mathbb{P}^1}

\newcommand{\ord}{\textrm{ord}}

\newcommand{\tmid}{\textrm{mid}}
\newcommand{\mc}{\textrm{mc}}
\newcommand{\Gr}{\textrm{Gr}\,}

\newcommand{\GM}{\textrm{GM}}
\newcommand{\MDR}{\textrm{MDR}}
\newcommand{\DR}{\textrm{DR}}
\newcommand{\Supp}{\textrm{Supp}\,}

\newcommand{\vb}{\mathrel{} \middle | \mathrel{}}

\newcommand{\xj}[1]{\Omega^{#1}_{X}(\log J)}
\newcommand{\zd}{\Omega^1_{Z}(D_Z)}

\newcommand{\xjzq}[1]{\Omega^{#1}_{(X,J)/(Z, D_Z)}}

\newcommand{\rMDR}{\textrm{MDR}_{X/Z}}
\newcommand{\rDR}{\textrm{DR}_{X/Z}}

\allowdisplaybreaks[1]

\begin{document}
\maketitle
\begin{abstract}
In this paper, we provide an explicit correspondence between moduli spaces of parabolic logarithmic connections on the projective line with spectral type $(2^2,1^4,1^4)$ and $(3^2, 2^3, 1^6)$, and the spaces of initial conditions for difference Painlev\'e equations with symmetry of type $E^{(1)}_7$ and $E^{(1)}_8$, respectively. Furthermore, we realize the extended affine Weyl group symmetries of the space of initial conditions in terms of connections.
\end{abstract}
	
\section{Introduction}

\subsection{The space of initial conditions for difference Painlev\'e equations and the moduli space of meromorphic connections}

Sakai \cite{Sa1} provided a geometric approach to the discrete
Painlev\'e equations. He characterized the good compactification of spaces of initial conditions for the Painlev\'e equations as a certain rational projective surface. Each such surface is obtained by blowing up
the projective plane $\mathbb{P}^2$ at 9 points, including infinitely near ones, and has a unique effective anti-canonical divisor $Y$. The divisor $Y$ is a singular elliptic curve. Hence, such a surface is classified by affine root system $R$ corresponding to the type of $Y$. We call a surface corresponding to $R$ \textrm{an $R$-surface} and denote it by $X(R)$. The extended affine Weyl group $\widetilde{W}(R^\perp)$ acts on a suitable family of $R$-surfaces. In Sakai's theory, the discrete Painlev\'e equations are discrete dynamical systems generated by the action of nonzero elements in the root lattice. In particular, $R$-surfaces are the spaces of initial conditions for the difference Painlev\'e equations with symmetry of type $R^\perp$.

The (additive) difference Painlev\'e equations are classified into the following eleven symmetry types:
\[
A^{(1)}_0, A^{(1)}_1, A^{(1)}_2, D^{(1)}_4, D^{(1)}_5, D^{(1)}_6, D^{(1)}_7, D^{(1)}_8, E^{(1)}_6, E^{(1)}_7, E^{(1)}_8.
\]
It is known that the difference Painlev\'e equations are obtained by gauge transformations of linear differential equations preserving their monodromy. In the case where $R^\perp=A^{(1)}_n$, $D^{(1)}_n$, the difference Painlev\'e equations are obtained by linear differential equations giving rise to Painlev\'e equations by isomonodromic deformation. In the case where $R^\perp=E^{(1)}_n$, Boalch \cite{Bo2} found corresponding Fuchsian systems from the
perspective of quiver varieties:
\[
\begin{minipage}{16cm}
	\centering
	\begin{tabular}{|c|c|c|c|c|}\hline
		surface type $(R)$& $A^{(1)**}_0$&$A^{(1)*}_1$&$A^{(1)*}_2$ \\ \hline
		symmetry type $(R^\perp)$& $E^{(1)}_8$&$E^{(1)}_7$&$E^{(1)}_6$ \\\hline
		spectral type&$3^2, 2^3, 1^6$&$2^2, 1^4, 1^4$&$1^3, 1^3, 1^3$ \\ \hline
	\end{tabular}
\end{minipage}
\]
Since the Fuchsian systems in these cases have ranks greater than two, their analysis is nontrivial.

Our interest is moduli space of connections. Inaba, Iwasaki, and Saito \cite{IIS1, In, IS1, IS2} show that moduli spaces of meromorphic connections coincide with spaces of initial conditions of nonliner differential equations obtained by isomonodromic deformations. Hence, $R$-surfaces are isomorphic to moduli spaces of certain meremorphic connections for $R^\perp=A^{(1)}_n$, $D^{(1)}_n$. The same picture are expected in the case where $R^\perp=E^{(1)}_n$. 

Moduli spaces of linear differential equations with singularities are realized as quiver varieties (see Boalch \cite{Bo1}, Crawley-Boevey \cite{Cr}, Hiroe and Yamakawa \cite{HY}, and Hiroe \cite{Hi}). Systems of linear differential equations are regarded as connections on the trivial bundle. Moduli spaces of linear differential equations are proper subset of that of connections. Hence, we have to consider not only linear differential equations but also connections to get $R$-surfaces. The moduli space of Fuchsian systems corresponding to type $R^\perp=E^{(1)}_n$ is
identified with the Kronheimer's ALE spaces of type $E^{(1)}_n$. Boalch \cite{Bo2} explained
how to obtain the surfaces with symmetry of type $E^{(1)}_n$ from the corresponding ALE
spaces, that is, how to partially compactify the moduli space of logarithmic connections on
the trivial bundle to get the full moduli space of logarithmic connections of degree zero.

One of our aims is to give an explicit correspondence between moduli spaces of logarithmic connections with given spectral type and the spaces of initial conditions of difference Painlev\'e equations. The author \cite{Ma1} provided an explicit correspondence for $R^\perp=E^{(1)}_6$. In this paper, we do that for $R^\perp=E^{(1)}_7, E^{(1)}_8$.

We state the first main result. For a set of spectral types $\bm{r}=(r^{(i)}_j)_{1\leq j\leq s_i}^{0\leq i\leq 2}$ and complex numbers $\nu_0, \nu_1, \nu_2$ with $\nu_0+\nu_1+\nu_2=-1$, set
\[
\mcN=\left\{\boldsymbol{\nu}=(\nu^{(i)}_j)_{1\leq j\leq s_i}^{0\leq i\leq 2} \in \mbC^{s_0+s_1+s_2}\vb \text{$\sum_{j=1}^{s_i}r^{(i)}_j\nu^{(i)}_j=\nu_i$ for all $i$ }\right\}.
\] 
Let $\mcM\rightarrow \mcN$ be the family of moduli spaces of $\boldsymbol{\alpha}$-stable parabolic logarithmic connections with spectral type $\bm{r}$ on the projective line $\pl$, whose fiber $M^{\boldsymbol{\alpha}}(\boldsymbol{\nu})$ at $\boldsymbol{\nu}\in \mcN$ is the moduli space of $\boldsymbol{\alpha}$-stable $\boldsymbol{\nu}$-parabolic logarithmic connections. Let $\mcS$ be the blow-up of $\mbP^2_{(x_0:x_1:x_2)}\times \mcN$ along $s_0+s_1+s_2$ sections
\[
(1: -\nu^{(0)}_j:0), \quad (1: \nu^{(1)}_j:1), \quad (0: -\nu^{(2)}_j:1).
\]
Let $\mcD$ be the strict transforms of $\{x_0(x_0-x_2)x_2=0\}\times \mcN$. We denote the fibers of $\mcS$ and $\mcD$ at $\boldsymbol{\nu}$ by $S_{\boldsymbol{\nu}}$ and $D_{\boldsymbol{\nu}}$, respectively. For $\bm{r}=(2^2, 1^4, 1^4)$, $(3^2, 2^3, 1^6)$, $S_{\boldsymbol{\nu}}\setminus D_{\boldsymbol{\nu}}$ is isomorphic to the space $X(R)$ of initial conditions for the difference Painlev\'e equation with symmetry of type $R^\perp=E^{(1)}_7$, $E^{(1)}_8$, respectively, without the effective anti-canonical divisor $Y(R)$. We note that $S_{\boldsymbol{\nu}}$ is not isomorphic to $X(R)$ (Section \ref{E7sec}, \ref{E8sec}).

\begin{theorem}(Theorem \ref{MT1}, \ref{icsE7}, \ref{icsE8})\label{mt1} 
For $\bm{r}=(2^2, 1^4, 1^4)$, $(3^2, 2^3, 1^6)$ and a suitable $\boldsymbol{\alpha}$, there is an isomorphism $\tilde{\psi}\colon \mcM \rightarrow \mcS\setminus \mcD$ over $\mcN$. In particular, moduli spaces of parabolic logarithmic connections with spectral type $\bm{r}=(2^2, 1^4, 1^4)$, $(3^2, 2^3, 1^6)$ are isomorphic to spaces of initial conditions of difference Painlev\'e equations with symmetry of type $E^{(1)}_7, E^{(1)}_8$, respectively, without the effective anti-canonical divisor. 
\end{theorem}

Finding a good coordinate on the moduli space of connections is important to describe the
difference Painelev\'e equations explicitly. Dzhamay, Sakai and Takenawa \cite{DST, DT} introduced rational parameters of Fuchsian systems corresponding to the type $R^\perp=D^{(1)}_4$, $E^{(1)}_6$, $E^{(1)}_7$, which provide a good coordinate on a Zariski open subset of moduli space of Fuchsian systems. They regarded 
$R$-surfaces as the surfaces obtained by blowing up $\pl \times \pl$ at eight points and gave an explicit
correspondence between Fuchsian systems and points on a Zariski open subset of $\pl \times \pl$.

The morphism $\tilde{\psi}$ in Theorem \ref{mt1} is constructed by using the apparent singularity and its dual parameter (Section \ref{appsec}, \ref{dualparasec}). We write down $\boldsymbol{\alpha}$-stable parabolic logarithmic connections on $\pl$ with spectral type $\bm{r}=(2^2, 1^4, 1^4), (3^2, 2^3, 1^6)$ by these parameters. The normal forms give us the explicit correspondence between
$\boldsymbol{\alpha}$-stable parabolic logarithmic connections and points on the space of initial conditions (Section \ref{normalform}). In particular, the apparent singularity and its dual parameter provide a coordinate on the moduli space of connections. This method extends the argument for the rank two case \cite{DL, IIS2, KS, KLSS} and the case where $\bm{r}=(1^3, 1^3, 1^3)$ \cite{Ma1} to the cases where $\bm{r}=(2^2, 1^4, 1^4)$, $(3^2, 2^3, 1^6)$.
	
\subsection{Realization of symmetries of difference Painlev\'e equations}

In Sakai's theory, the extended affine Weyl group $\widetilde{W}(R^\perp)$ acts on a family of $X(R)$. Hence, $\widetilde{W}(R^\perp)$ also acts on the family of moduli spaces $M^{\boldsymbol{\alpha}}(\boldsymbol{\nu})$. This action should be written in terms of connections. In the case of Fuchsian systems, Boalch \cite{Bo2} explained that this action is obtained by additions and middle convolutions from the point of view of quiver varieties. On the other hand, it is nontrivial that the apparent singularity and its dual parameter of logarithmic connections are compatible with middle convolution. If it is true, then we can say that  the difference Painlev\'e equations are invariant under middle convolution by Theorem \ref{mt1}. We note that the commutativity of the isomonodromic deformation and middle convolution was established by Haraoka and Filipuk \cite{HF}, Boalch \cite{Bo3}, and Yamakawa \cite{Ya}. 

Our second aim is to realize symmetries of difference Painlev\'e equations in terms of connections. The author \cite{Ma2} did it in the case where $R^\perp=E^{(1)}_6$. In this paper, we realize symmetries of difference Painlev\'e equations for $R^\perp=E^{(1)}_7$, $E^{(1)}_8$. For the reflections at the non-central nodes and the automorphisms of the Dynkin diagram, their realization is straightforward. On the other hand, it is not clear how to realize the reflection at the central node.

We state the second main theorem. Let $w$ be the reflection at the central node of the Dynkin diagram, and let $\phi_{w} \colon \mcS \rightarrow \mcS$ denote the corresponding isomorphism. We will see in Section \ref{actS} that $\phi_{w}$ is induced by a quadratic transformation of the projective plane $\mbP^2$. Let $\tilde{\psi}\colon \mcM \rightarrow \mcS$ be the isomorphism in Theorem \ref{mt1}. We denote the middle convolution with a meromorphic one form $\beta$ by $\mc_\beta$ (which will be explained in Section \ref{mcsec}). Let $\mcN^0$ be the open subset of $\mcN$ defined in Section \ref{realizecn}, and set $\mcM^0=\mcM\times_{\mcN}\mcN^0$ and $\mcS^0=\mcS\times_{\mcN}\mcN^0$.

\begin{theorem}[Theorem \ref{MT2}]\label{mt2}
The following diagram commutes for a suitable $\beta$.
\[
\begin{tikzcd}
	\mcM^0\arrow[r, " \mc_\beta "]\arrow[d, "\tilde{\psi}"]&\mcM^0\arrow[d, "\tilde{\psi}"]\\
	\mcS^0\arrow[r, "\phi_{w}"]&\mcS^0
\end{tikzcd}
\]
\end{theorem}
This theorem means that the apparent singularity and its dual parameter are compatible with middle convolution. In this sense, these parameters give a good coordinate on the moduli space of parabolic logarithmic connections.

We note that the same assertions as in Theorem \ref{mt1}, \ref{mt2} hold when $R^\perp=E^{(1)}_6$ \cite{Ma1, Ma2}.

\subsection{Outline of this paper}

This paper is organized as follows. In Section 2, we recall the basic definitions and facts of parabolic connections. In section 3, we define the apparent singularity of logarithmic connections and we construct the morphism $\psi \colon \mcM \rightarrow \mbP^2$. In this process, we show that the underlying bundles are isomorphic to $\mcO(1)\oplus \mcO^{r-1}$ for a suitable $\boldsymbol{\alpha}$. In section 4, we prove Theorem \ref{mt1}. We give a normal form of logarithmic connections, and we show that the family $\mcM$ of moduli spaces is isomorphic to the family $\mcS \setminus \mcY$ of the surfaces without the effective anti-canonical divisors. After that, we provide an explicit correspondence between $\mcS$ and the family $\mcX$ of the $R$-surfaces. In Section 5, we explain how to define the action of $\widetilde{W}(R^\perp)$ on the family $\mcS\rightarrow \mcN$. After recalling the extended affine Weyl group $\widetilde{W}(R^\perp)$, we explain the action of $\widetilde{W}(R^\perp)$ on the Picard group of an $R$-surface. This action leads to the actions on  $\mcN$ and $\mcS$. We summarize the explicit form of the action in appendices. In Section 6, we realize the action on $\mcS$ as transformations of parabolic connections.We review the theory of middle convolution by Simpson \cite{Sim} and explain the main theorem \ref{mt2}. In appendices, we summarize entries of normal form, explicit forms of actions of $\widetilde{W}(R^\perp)$, and realization of the actions.

\section*{Acknowledgments}
The author is deeply grateful to Masa-Hiko Saito and Szilard Szab\'o for sharing their unpublished theory on the apparent singularities and for providing valuable comments on this paper. This work is supported by Japan Society for the Promotion of Science KAKENHI Grant Numbers 22H00094. 
		
\section{Parabolic logarithmic connections with given spectral type and moduli spaces}
In this section, we recall the moduli theory of parabolic connections 
with given spectral tupe due to Inaba-Saito \cite{IS2}. 

Take $n$ distinct points $t_1, \ldots, t_n \in \pl$. Set $D=[t_1]+\cdots+[t_n]$ and $L=\Omega_{\pl}^1(D)$. We denote the structure sheaf of $\pl$ by $\mcO$. For a point $q \in \pl$ and a sheaf $E$, we denote the stalk of $E$ at $q$ by $E_{q}$. 
\begin{definition}
A \textit{logarithmic connection} $(E,\nabla)$ of rank $r$ and degree $d$ over $(\pl, D)$ is a pair of a vector bundle $E$ of rank $r$ and degree $d$ on $\pl$ and a $\mathbb{C}$-homomorphism $\nabla \colon E\rightarrow E\otimes L$ satisfying Leibniz rule, i.e.
\[
\nabla(fa)=a\otimes df+f\nabla(a)
\]
for any $f\in \mcO$ and $a\in E$. 
\end{definition}
Let us fix an integer $d$ and a positive integer $r$. Take positive integers $r^{(i)}_{1}, \ldots, r^{(i)}_{s_i}$ with $\sum_{j=1}^{s_i}r^{(i)}_{j}=r$, and set $\bm{r}:=(r^{(i)}_{j})^{1\leq i\leq n}_{1\leq j\leq s_i}$. We denote the residue field at $t_i$ by $k(t_i)$ and set $E|_{t_i}:=E\otimes k(t_i)$. $E|_{t_i}$ is a vector space over $\mbC$ of dimension $r$. Let $\res_{t_i}\nabla$ be the endomorphism of $E|_{t_i}$ induced by $\nabla$, and we call $\res_{t_i} \nabla$ the \textit{residue endomorphism} at $t_i$.
\begin{definition}\label{pcdef}
Take a set of complex numbers $\boldsymbol{\nu}=(\nu_{j}^{(i)})^{1\leq i\leq n}_{1\leq j \leq s_i}$. A \textit{$\boldsymbol{\nu}$-parabolic logarithmic connection} $(E,\nabla, l_*=\{l_{*}^{(i)}\}_{1 \leq i\leq n})$ over $(\pl, D)$ with spectral type $\bm{r}$ is a collection consisting of the following date;
\begin{itemize}\setlength{\itemsep}{0cm}
\item[(1)] $(E, \nabla)$ is a logarithmic connection of rank $r$ and degree $d$ over $(\pl, D)$, 
\item[(2)] $l_{*}^{(i)}$ is a filtration $E|_{t_i}=l_{0}^{(i)} \supsetneq \cdots \supsetneq  l_{s_i-1}^{(i)} \supsetneq l_{s_i}^{(i)}=\{0\}$ by subspaces such that $\dim_{\mathbb{C}}l_{j}^{(i)}/l_{j+1}^{(i)}=r^{(i)}_{j+1}$ and 
\[
(\res_{t_i}(\nabla)-\nu_{j+1}^{(i)}\id )({l_j^{(i)}})\subset l_{j+1}^{(i)}
\]
for $1\leq i\leq n$ and $0\leq j\leq r-1$.	
\end{itemize}	
\end{definition}

The condition (2) of Definition \ref{pcdef} implies that $\nu^{(i)}_{1}, \ldots, \nu^{(i)}_{s_i}$ are the eigenvalues of the residue endomorphism $\res_{t_i}\nabla$, and $\res_{t_i}\nabla$ can be represented for a suitable basis by
\[
\res_{t_i}\nabla=
\begin{pmatrix}
	\nu^{(i)}_{s_i}I_{r^{(i)}_{s_i}}&*&*\\
	0&\ddots&*\\
	0&0&\nu^{(i)}_{1}I_{r^{(i)}_1}
\end{pmatrix}.
\]
We call $\nu^{(i)}_{j}$ a \textit{local exponent} of $\nabla$ at $t_i$.
If $\nu^{(i)}_{1}, \ldots, \nu^{(i)}_{s_i}$ are different, then $l^{(i)}_{j}$ is the direct sum of the eigenspaces of the eigenvalues $\nu^{(i)}_{j+1}, \ldots, \nu^{(i)}_{s_i}$.

Let $\boldsymbol{\alpha}=\{\alpha_j^{(i)}\}^{1 \leq i \leq n}_{1 \leq j \leq s_i}$ be a set of rational numbers satisfying $0 < \alpha_{1}^{(i)}<\cdots <\alpha_{s_i}^{(i)} < 1$ for each $i=1,\ldots,n$ and $\alpha_{j}^{(i)}\neq \alpha_{j'}^{(i')}$ for $(i,j)\neq (i',j')$. 

\begin{definition}
A $\boldsymbol{\nu}$-parabolic connection $(E,\nabla,l_*)$ with spectral type $\bm{r}$ is said to be \textit{$\boldsymbol{\alpha}$-stable} if for any nonzero subbundle $F \subsetneq E$ preserved by $\nabla$, the inequality 
\[
\frac{\deg F + \sum_{i=1}^{n} \sum_{j=1}^{s_i} \alpha_{j}^{(i)} \dim ((F|_{t_i} \cap l_{j-1}^{(i)})/(F|_{t_i} \cap l_{j}^{(i)}))}{\rank F}<\frac{\deg E + \sum_{i=1}^{n} \sum_{j=1}^{s_i} \alpha_{j}^{(i)}}{\rank E}
\]
holds.
\end{definition}

A $\boldsymbol{\nu}$-parabolic connection $(E,\nabla, l_*=\{l_*^{(i)}\}_{1 \leq i\leq n})$ satisfies the relation
\[
\sum_{i=1}^n\sum_{j=1}^{s_i}r^{(i)}_j\nu_{j}^{(i)}+\deg E=0,
\]
which called the Fuchs relation. Set
\[
\mcN(d, \bm{r}):=\left\{
(\nu^{(i)}_{j})^{1\leq i\leq n}_{1\leq j\leq s_i} \in \mbC^{\sum_{i=1}^ns_i} \vb \sum_{i=1}^n\sum_{j=1}^{s_i}r^{(i)}_{j}\nu^{(i)}_{j}+d=0
\right\}.
\]

\begin{theorem}\label{modpc}(Theorem 1.1, 1.2, 1.3 in \cite{IS2})
	There exists a relative coarse moduli scheme
	\[
	\colon \mcM^{\boldsymbol{\alpha}}(d, \bm{r})\longrightarrow \mcN(d, \bm{r})
	\]
	of $\boldsymbol{\alpha}$-stable parabolic connections of rank $r$ and degree $d$, which is smooth and quasi-projective, and whose fiber $M^{\boldsymbol{\alpha}}(\boldsymbol{\nu})$ over $\boldsymbol{\nu} \in \mcN(d, \bm{r})$ is the moduli space of $\boldsymbol{\alpha}$-stable $\boldsymbol{\nu}$-parabolic connections  over $(\pl, D)$. The moduli space $M^{\boldsymbol{\alpha}}(\boldsymbol{\nu})$ is of equidimension
	\[
	2\sum_{i=1}^n\sum_{j=1}^{s_i}\sum_{j'>j}r^{(i)}_{j}r^{(i)}_{j'}-2r^2+2.
	\]
	if $M^{\boldsymbol{\alpha}}(\boldsymbol{\nu})\neq \emptyset$.
\end{theorem}
We denote the spectral types $(22,1111,1111)$ and $(33, 222, 111111)$ by $(2^2, 1^4, 1^4)$ and  $(3^2, 2^3, 1^6)$, respectively. When $\bm{r}=(2^2, 1^4, 1^4), (3^2, 2^3, 1^6)$, the dimension of the moduli space $M^{\boldsymbol{\alpha}}(\boldsymbol{\nu})$ is two. For a set of spectral types $\bm{r}=(r^{(i)}_j)_{1\leq j\leq s_i}^{0\leq i\leq 2}$ and complex numbers $\nu_0, \nu_1, \nu_2$ with $\nu_0+\nu_1+\nu_2=-1$, let
\[
\mcN=\left\{\boldsymbol{\nu}=(\nu^{(i)}_j)_{1\leq j\leq s_i}^{0\leq i\leq 2} \in \mbC^{s_0+s_1+s_2}\vb \text{$\sum_{j=1}^{s_i}r^{(i)}_j\nu^{(i)}_j=\nu_i$ for all $i$ }\right\}
\] 
and $\mcM=\mcM^{\boldsymbol{\alpha}}(1, \bm{r})\times_{\mcN(1, \bm{r})}\mcN$. One of our aim is to show that $\mcM$ for $\bm{r}=(2^2, 1^4, 1^4), (3^2, 2^3, 1^6)$ is isomorphic to a family of the spaces of initial conditions for the difference Painlev\'e equations with symmetry of type $E^{(1)}_7,  E^{(1)}_8$, respectively.

\section{Construction of the morphism from the moduli spaces of parabolic connections to the projective plane}
In this section, we construct a morphism $\psi \colon \mcM\to \mbP^2$ for $D=[0]+[1]+[\infty]$ and $\bm{r}=(2^2, 1^4, 1^4), (3^2, 2^3, 1^6)$. Let $\pi \colon \mbL \to \pl$ be the total space of $L=\Omega_{\pl}^1(D)$, that is, $\mbL=\Spec \Sym(L^\vee)$. We define a map $\App \colon \mcM\rightarrow \pl$ using the theory of apparent singularities, and we lift $\App$ to $\varphi \colon \mcM\rightarrow \mbL$. We obtain a morphism $\psi$ by composing $\varphi$ with a morphism $\mbL \rightarrow \mbP^2$. 

\subsection{Apparent singularities} \label{appsec}
The theory of apparent singularities for logarithmic connections are defined by Saito and Szab\'o \cite{SS}. In this paper, we adopt a slightly modified version of their theory, which is essentially the same as theirs. In this section, we do not assume that $D=[0]+[1]+[\infty]$ and $\bm{r}=(2^2, 1^4, 1^4), (3^2, 2^3, 1^6)$.

We take a logarithmic connection $(E, \nabla)$ over $(\pl, D)$ of rank $r$ and degree 1. By Riemann-Roch theorem, the Euler characteristic  $\chi\left(\mcO(1)^\vee \otimes E\right)$ is one, which implies $\dim \Hom(\mcO(1), E)\geq 1$. Hence, there is a nonzero homomorphism $\sigma \colon \mcO(1) \rightarrow E$. Let us fix a $\sigma$. 

\begin{assumption}\label{appfil}
	For a collection $(E, \nabla, \sigma)$, there exists a unique filtration 
	\[
	E\supsetneq F^0\supsetneq F^1\supsetneq F^2\supsetneq \cdots \supsetneq F^{r-1}\supsetneq F^r=\{0\}
	\]
	by locally free sheaves satisfying the following properties.
	\begin{itemize}
		\item[(1)] $F^{r-1}=\Image(\sigma)$.
		\item[(2)] $\rank F^j=r-j$.
		\item[(3)] $\nabla(F^j)\subset F^{j-1}\otimes L$ for $1\leq j\leq r-1$ and the induced $\mcO$-homomorphism 
		\[
		\Gr^j(\nabla)\colon F^{j}/F^{j+1}\longrightarrow (F^{j-1}/F^{j})\otimes L
		\]
		is an isomorphism.
		\item[(4)] $F^{r-j-1}/F^{r-j}\cong L^{-j}\otimes \mcO(1)$ for $1\leq j\leq r-1$.
	\end{itemize}
\end{assumption}

Under Assumption \ref{appfil}, we define a torsion sheaf $T(\nabla, \sigma)$ by the following exact sequence
\[
0\longrightarrow  F^0\overset{\nabla}{\longrightarrow} E \longrightarrow T(\nabla, \sigma) \longrightarrow 0.
\]
Since
\[
\deg F^j=\sum_{k=0}^{r-j-1}\deg \left(L^{-k}\otimes \mcO(1)\right)=-\frac{1}{2}(r-j)(r-j-1)(n-2)+(r-j), 
\]
we have
\[
\deg T(\nabla, \sigma)=1-\left(-\frac{1}{2}r(r-1)(n-2)+r\right)=\frac{1}{2}r(r-1)n+1-r^2=:N.
\]
We note that $N$ is the half of the dimension of the moduli space $M^{\boldsymbol{\alpha}}(\boldsymbol{\nu})$ for $\bm{r}=(1^r, \ldots, 1^r)$. 
\begin{definition}
We call the effective divisor
\[
\App(E, \nabla, \sigma):=\sum_{q\in \pl}m_q[q]\quad (m_{q}:=\length_{\mcO_{q}} T(\nabla, \sigma)_{q})
\]
the \textit{apparent singularity divisor} of $(E, \nabla, \sigma)$. 
We call a point $q$ an \textit{apparent singularity} of $(E, \nabla, \sigma)$ if $m_{q}>0$. 
\end{definition}

We note that Saito and Szab\'o \cite{SS} showed that a collection $(E, \nabla, \sigma)$ satisfy Assumption \ref{appfil} if $(E, \nabla)$ is irreducible. In particular, they defined the apparent singularities for pairs of irreducible connections $(E, \nabla)$ and sections $\sigma$. On the other hand, we have to define the apparent singularities for all $\boldsymbol{\alpha}$-stable parabolic connections to describe the whole of moduli spaces. In the next section, we show that all $\boldsymbol{\alpha}$-stable parabolic connections satisfy Assumption \ref{appfil} for $\bm{r}=(2^2, 1^4, 1^4), (3^2, 2^3, 1^6)$ and a suitable $\boldsymbol{\alpha}$. The remainder of this section is devoted to proving a lemma and a proposition used in the next section.

\begin{lemma}\label{saturation}
	Let $(E, \nabla)$ be a connection. For subsheaves $F_1\subset F_2\subset E$, we denote their saturation by $E_1, E_2$. If $\nabla(F_1)\subset F_2\otimes L$, then $\nabla(E_1)\subset E_2\otimes L$. 
\end{lemma}
\begin{proof}
	Since $F_1\subset F_2$, we have $E_1\subset E_2$. In particular, the composition
	\[
	u\colon E_1\overset{\nabla}{\longrightarrow} E\otimes L\longrightarrow E/E_2\otimes L
	\]
	is an $\mcO$-homomorphism. Since $u(F_1)=\{0\}$ by definition, $u$ induces an $\mcO$-homomorphism $\bar{u}\colon E_1/F_1\to E/E_2\otimes L$. Since rank of $E_1/F_1$ is zero and $E/E_2$ is locally free, $\bar{u}$ must be zero. Hence, we have $\nabla(E_1)\subset E_2\otimes L$. 
\end{proof}
We consider a collection $(E, \nabla, \sigma)$ satisfying Assumption \ref{appfil}. Let $E^j$ be the saturation of $F^j$. For $j=0, \ldots, r-2$, we define an $\mcO$-homomorphism $u^{(j)}$ by
\begin{equation}\label{uj}
u^{(j)} \colon E^{j+1} \overset{\nabla}{\longrightarrow} E^{j} \otimes L\longrightarrow E^{j}/E^{j+1} \otimes L.
\end{equation}
We denote the cokernel $\Coker u^{(j)}$ by $T^{(j)}$. Since $u^{(j)}$ is not zero by Lemma \ref{saturation} and Assumption \ref{appfil} (3), the rank of $T^{(j)}$ is zero. Let $T^{(r-1)}:=E^{r-1}/F^{r-1}$. 
\begin{proposition}\label{mq}
For $(E,\nabla, \sigma)$ and $q\in \pl$, we have
\[
m_{q}=\sum_{j=1}^{r} j\cdot \length_{\mcO_{q}}T^{(j-1)}_q.
\]
In particular, if $\length_{\mcO_{q}}T^j_q>0$ for some $j$, then $q$ is an apparent singularity.
\end{proposition}
\begin{proof}
	Let $z \in \mcO_q$ be a local parameter. We can take sections $f_1, \ldots, f_{r}$ of $E$ by Assumption \ref{appfil} (2) such that $f_j, \ldots, f_{r}$ generate $F^{j-1}$ and
	\begin{equation}\label{fcond}
		\nabla(f_{j+1})-f_{j}\otimes \omega\in F^{j}\otimes L
	\end{equation}
	for $1\leq j\leq r-1$. Here, $\omega=\frac{dz}{z}$ if $q=t_1, \ldots, t_n$, and $\omega=dz$ otherwise. We take a basis $e_1, \ldots, e_{r}$ of $E$ such that $e_{j}, \ldots e_{r}$ generate $E^{j-1}$ for $1\leq j\leq r$.  Then, there is a $b_{r}\in \mcO_q$ such that $f_{r}=b_{r}e_{r}$. Since $\nabla(E^{j})\subset E^{j-1}\otimes L$ for $1\leq j\leq r-1$, there is an element $b_{j} \in \mcO_q$ such that
	\begin{equation}\label{econd}
		\nabla(e_{j+1})-b_{j}e_{j}\otimes \omega\in E^{j}\otimes L.
	\end{equation}
	Since
	\begin{align*}
		\nabla(f_{r}-b_{r}e_{r})=&\left(\nabla(f_{r})-f_{r-1}\otimes \omega\right)-b_{r}\left(\nabla(e_{r})-b_{r-1}e_{r-1}\otimes \omega\right)-e_{r}\otimes db_{r}\\
		&+f_{r-1}\otimes \omega-b_{r-1}b_{r}e_{r-1}\otimes \omega,
	\end{align*}
	we have $f_{r-1}-(b_{r-1}b_{r})e_{r-1} \in E^{r-1}$. By repeating this calculation, we have $f_{j}-(b_{j}\cdots b_{r})e_{j}\in E^{j}$. Hence, we obtain $(f_{r}, \cdots, f_1)=
	(e_{r}, \cdots, e_1)B$, 
	\[B=
	\begin{pmatrix}
		b_{r}&*&\cdots&*\\
		0&\ddots &\ddots&\vdots\\
		\vdots&\ddots&\hspace{1mm}\ddots&*\\
		0&\cdots&0&\hspace{-1mm}b_1\cdots b_{r}
	\end{pmatrix} \in M_r(\mcO_q)
	\]
	We denote the order $f \in \mcO_q$ of $z$ by $\ord_z(f)$. Since
	\[
	\ord_z(b_j)=\length_{\mcO_{q}}T^{(j-1)}_q,
	\]
	we have
	\[
	m_q=\ord_z(\det B)=\sum_{j=1}^{r}j\cdot \ord_z(b_j)=\sum_{j=1}^{r}j\cdot\length_{\mcO_{q}}T^{(j-1)}_q.
	\]
\end{proof}
\subsection{Underlying bundle}
In this section, we show that the underlying bundles of $\boldsymbol{\alpha}$-stable parabolic connections on $(\pl, D=[0]+[1]+[\infty])$ are isomorphic to $\mcO(1)\oplus \mcO^{\oplus r-1}$ for  $\bm{r}=(2^2, 1^4, 1^4), (3^2, 2^3, 1^6)$ and a suitable $\boldsymbol{\alpha}$.

First, we show that $\boldsymbol{\alpha}$-stable parabolic connections satisfying Assumption \ref{appfil} for $\bm{r}=(2^2, 1^4, 1^4)$.

\begin{lemma}
	Assume that $0<\alpha^{(i)}_j \ll 1$ for all $i, j$. Let $(E, \nabla, l_*)$ be an $\boldsymbol{\alpha}$-stable $\boldsymbol{\nu}$-parabolic connection with spectral type $(2^2, 1^4, 1^4)$ and degree $1$. We take a nonzero homomorphism $\sigma \colon \mcO(1)\to E$. Then, $(E, \nabla, \sigma)$ satisfies Assumption \ref{appfil}. 
\end{lemma}
\begin{proof}
	Let $F^3:=\Image(\sigma)$. We inductively define a homomorphism $v^{(i)}$ and a subseaf $F^i$ as follows. Let $v^{(i)}$ be the composition
	\[
	F^i\overset{\nabla}{\longrightarrow}E\otimes L \longrightarrow E/F^{i} \otimes L
	\]
	and $F^{i}$ be the kernel of the composition of the quotients
	\[
	E\longrightarrow E/F^{i+1} \longrightarrow (E/F^{i+1} )/\left(\Image(v^{(i+1)})\otimes L^{-1}\right).
	\]
	We can see that $F^i$ satisfies Assumption \ref{appfil} if $v^{(i+1)}\neq 0$. Hence, It is sufficient for all $i$ to show that $v^{(i)}\neq 0$, that is, $F^i$ is not invariant under $\nabla$. If $\nabla(F^i)\subset F^i\otimes L$ for $i=2, 3$, then $F^i$ destabilizes $(E, \nabla, l_*)$ since
	\[
	\frac{\deg F^i}{\rank F^i}-\frac{\deg E}{\rank E}\geq  \frac{1}{4}> \frac{\sum_{i=2}^{0} \sum_{j=1}^{s_i} \alpha_{j}^{(i)}}{\rank E}-\frac{ \sum_{i=0}^{2} \sum_{j=1}^{s_i} \alpha_{j}^{(i)} \dim ((F|_{t_i} \cap l_{j-1}^{(i)})/(F|_{t_i} \cap l_{j}^{(i)}))}{\rank F}.
	\]
	Hence, we have $\nabla(F^i)\nsubseteq F^i\otimes L$ for $i=2, 3$.
	
	Suppose that $F^1$ is invariant under $\nabla$. Let $E^i$ be the saturation of $F^i$. Then, $E^1$ is also invariant under $\nabla$ by Lemma \ref{saturation}. We show that $E^1$ destabilizes $(E, \nabla, l_*)$. We take a basis $e_1, e_2, e_3$ of $E^1|_{0}$ such that $e_j, \ldots, e_3$ generate $E^j|_{0}$ for all $j$. Then, $\res_{0}\nabla|_{E^1}$ is of the form
	\[
	(\res_{0}\nabla|_{E^1})(e_3, e_2, e_1)=(e_3, e_2, e_1)A, \quad A=
	\begin{pmatrix}
		*&*&*\\a&*&*\\0&b&*	
	\end{pmatrix}.
	\]
	When $\nu^{(0)}_1 \neq \nu^{(0)}_2$, since both of the eignespaces of $\res_{0}\nabla$ for $\nu^{(0)}_1$ and $\nu^{(0)}_2$ are two dimentional, either eigenspace is included in $E^1|_0$. When $\nu^{(0)}_1=\nu^{(0)}_2$, we can see that $E^1|_0$ contains a two-dimensional subspace of the eigenspace for $\nu^{(0)}_1$ by the definition of $l^{(0)}_*$. In either case, all $2\times 2$ minors of $A$ vanish. Hence, at least one of $a$ and $b$ is zero, and we have $F^1 \nsubseteq E^1$. From this, we have $\deg E^1\geq \deg F^1+1=1$, and
	\[
	\frac{\deg E^1}{\rank E^1}-\frac{\deg E}{\rank E}\geq \frac{1}{12}> \frac{\sum_{i=1}^{3} \sum_{j=1}^{s_i} \alpha_{j}^{(i)}}{\rank E}-\frac{ \sum_{i=1}^{3} \sum_{j=1}^{s_i} \alpha_{j}^{(i)} \dim ((F|_{t_i} \cap l_{j-1}^{(i)})/(F|_{t_i} \cap l_{j}^{(i)}))}{\rank F}.
	\]
	Hence, $E^1$ destabilizes $(E, \nabla, l_*)$. Therefore, $F^1$ is not invariant under $\nabla$.
\end{proof}

\begin{proposition}\label{bdlE7}
Assume that $0<\alpha^{(i)}_j \ll 1$ for all $i, j$. Let $(E, \nabla, l_*)$ be an $\boldsymbol{\alpha}$-stable $\boldsymbol{\nu}$-parabolic connection with spectral type $(2^2, 1^4, 1^4)$ and degree $1$. Then, $E$ is isomorphic to $\mcO(1)\oplus \mcO^{\oplus 3}$.
\end{proposition}
\begin{proof}
We take a nonzero homomorphism $\sigma \colon \mcO(1) \rightarrow E$. Let $\{F^i\}$ be the filtration of subseaves of $E$ satisfying Assumption \ref{appfil}. Let $E^i$ be the saturation of $F^i$. We take a basis $e_1, e_2, e_3$ of $E^1|_{0}$ such that $e_j, \ldots, e_3$ generate $E^j|_{0}$ for all $j$. Then, $\res_{0}\nabla$ is of the form
\[
(\res_{0}\nabla)(e_3, e_2, e_1, e_0)=(e_3, e_2, e_1, e_0)
\begin{pmatrix}
	*&*&*&*\\a&*&*&*\\0&b&*&*\\
	0&0&c&*
\end{pmatrix}.
\]
When $\nu^{(0)}_1 \neq \nu^{(0)}_2$, since both of the eigenspaces for $\nu^{(0)}_1$ and $\nu^{(0)}_2$ are two-dimensional, we have $b=0$ or $a=c=0$. When $\nu^{(0)}_1 =\nu^{(0)}_2$, we also have $b=0$ or $a=c=0$ by the definition of $l^{(0)}_*$. Suppose that $a=c=0$. Then, we have $m_0\geq 4$ by Proposition \ref{mq}, which contradicts to $m_0\leq N=3$. Hence, $b$ must be zero. By Proposition \ref{mq}, $\App(E, \nabla, \sigma)$ is of the form $2[0]+[q]$. Then, we have $T^{(2)}=T^{(3)}=0$, $T^{(1)}=k(0)$ and $T^{(0)}=k(q)$, where $T^{(i)}$ is a torsion sheaf defined in the previous section. Hence, we have $E^3=F^3\cong \mcO(1)$ and $E^2=F^2\cong \mcO(1)\oplus \mcO$. By the definition of $T^{(1)}$,  the short exact sequence
\[
0\longrightarrow E^2/E^3\otimes L^{-1}\longrightarrow E^1/E^2 \longrightarrow T^1\otimes L^{-1} \longrightarrow 0.
\]
From this, we have $E^1/E^2 \cong \mcO$. Hence, we have $E^1\cong \mcO(1)\oplus \mcO^{\oplus 2}$. In a similar manner, we have $E\cong \mcO(1)\oplus \mcO^{\oplus 3}$. 
\end{proof}

We have the following from Proposition \ref{bdlE7}.
\begin{corollary}\label{appE7}
Under the same setting, the apparent singularities of $(E, \nabla, l_*)$ are uniquely determined, and the apparent singularity divisor is of the form $2[0]+[q]$. 
\end{corollary}

We have the same description for the $E^{(1)}_8$-case.
\begin{proposition}\label{bdlE8}
	Assume that $0<\alpha^{(i)}_j \ll 1$ for all $i, j$. Let $(E, \nabla, l_*)$ be an $\boldsymbol{\alpha}$-stable parabolic connection with spectral type $(3^2, 2^3, 1^6)$ and degree $1$. Then, $E$ is isomorphic to $\mcO(1)\oplus \mcO^{\oplus 5}$. 
\end{proposition}
In the proof of the above proposition, we have $T^{(0)}=k(q)$, $T^{(1)}=T^{(3)}=k(0)$, $T^{(2)}=k(1)$, and $T^{(4)}=T^{(5)}=0$. In particular, we have the following.
\begin{corollary}\label{appE8}
	Under the same setting, the apparent singularities of $(E, \nabla, l_*)$ are uniquely determined, and the apparent singularity divisor is of the form $6[0]+3[1]+[q]$. 
\end{corollary}

\subsection{Apparent map}
Since the apparent singularity divisor of $(E, \nabla, l_*)$ uniquely determined by Corollary \ref{appE7} and \ref{appE8}, we denote $\App(E, \nabla, \sigma)$ by $\App(E, \nabla)$. Moreover, we have $\App(E, \nabla)=2[0]+[q]$ for $R^\perp=E^{(1)}_7$ and $\App(E, \nabla)=6[0]+3[1]+[q]$ for $R^\perp=E^{(1)}_8$. Hence, the apparent singularity divisor induces a morphism
\[
\App \colon \mcM\longrightarrow \pl, \quad (E, \nabla, l_*) \longmapsto q
\]
We call the morphism $\App$ the \textit{apparent map}.

\subsection{Construction of the morphism from the moduli space to the projective plane}\label{dualparasec}
In this section, we construct a morphism $\psi \colon \mcM\to \mbP^2$. First, we construct a morphism $\varphi \colon \mcM\to \mbL$. Let us fix a decomposition $E\cong \mcO([\infty])\oplus \mcO^{\oplus r-1}$. Let $d_{\infty} \colon \mcO([\infty]) \rightarrow \mcO([\infty])\otimes \Omega_{\pl}^1([\infty])$ and $d\colon \mcO\rightarrow \mcO \otimes  \Omega_{\pl}^1$ be the canonical connections, and set $\nabla_0=d_{\infty}\oplus d^{\oplus r-1}$. We define the $\mcO$-homomorphism $B$ by
\[
g \colon E\overset{\nabla-\nabla_0}{\longrightarrow} E\otimes L\longrightarrow E/E^1\otimes L.
\]
We note that $g$ does not depend on choices of a decomposition of $E$. Let $q=\App(E, \nabla)$. Since $u^{(0)}|_q=0$, where $u^{(0)}$ is defined in (\ref{uj}), we have $g|_q=0$. Hence, the homomorphism $g$ induces an $\mcO$-homomorphism $h\colon (E/E^{1})|_{q} \rightarrow (E/E^{1}\otimes L)|_q$.
\begin{equation}
	\begin{tikzcd}
		0 \arrow[r]&E^1|_{q}\arrow[r]\arrow[rd,"0"]&E|_{q} \arrow[r]\arrow[d, "g|_q"]&(E/E^1)|_{q}\arrow[r] \arrow[ld, "h"]&0 \\
		&&(E/E^1\otimes L)|_q&&
	\end{tikzcd}
\end{equation}
The map $h$ induces a linear map $h^\vee \colon (L^\vee)|_q \rightarrow\End\, (E/E^{j})|_q \cong \mbC$. Let $\mcJ$ be the kernel of the map
\[
\Sym(L^\vee) \longrightarrow \Sym(L^\vee)|_q \overset{h^\vee}{\longrightarrow} \mbC.
\]
Then, the image
\[
\mcI:=\Image (\pi^*\mcJ\to \pi^*\pi_*\mcO_{\mbL}\to \mcO_{\mbL})
\]
is an ideal sheaf of $\mcO_{\mbL}$. Here, we note that  $\pi_*\mcO_{\mbL} \cong \Sym(L^\vee)$. Let $Z:=\Supp (\mcO_{\mbL}/\mcI)$ .

We define a morphism $\varphi$ by
\[
\varphi\colon \mcM\to \mbL, \quad (E, \nabla, l_*)\mapsto Z.
\]
We consider this morphism in terms of coordinates. Let $U_0=\pl\setminus \{\infty\}$ and $U_\infty=\pl \setminus \{0\}$. Let $z, w$ be coordinates of $U_0$ and $U_\infty$, respectively, with $w=z^{-1}$. Then, $\tfrac{dz}{z(z-1)}$ and $\tfrac{dw}{w(w-1)}$ be a section of $L$ on $U_0$ and $U_\infty$, respectively, and we have isomorphisms
\[
\mbC^2 \overset{\sim}{\rightarrow} \pi^{-1}(U_0);  \quad (q, p)\mapsto (q, p\tfrac{dz}{z(z-1)}),  \quad  \mbC^2 \overset{\sim}{\rightarrow} \pi^{-1}(U_\infty); \quad (\tilde{q}, \tilde{p})\mapsto (\tilde{q}, \tilde{p}\tfrac{dw}{w(w-1)}).
\]
Here, coordinates $q, p, \tilde{q}, \tilde{p}$ satisfy $\tilde{q}=q^{-1}$ and $\tilde{p}=q^{-1}p$. Then, $\varphi(E, \nabla, l_*)$ can be written in the form $(q, p\tfrac{dz}{z(z-1)})$ or $(\tilde{q}, \tilde{p}\tfrac{dw}{w(w-1)})$, and accordingly we have $\App(E, \nabla, l_*)=q$ or $\App(E, \nabla, l_*)=\tilde{q}$, respectively. We call the parameter $p$ and $\tilde{p}$ the \textit{dual parameter} of $q$ and $\tilde{q}$, respectively.

Next, we construct a morphism $\psi \colon \mcM\to \mbP^2$.
 Let $\mbL \rightarrow \mbP(\mcO\oplus L)$ be the inclusion and $\mbP(\mcO\oplus L)\rightarrow \mbP^2$ be the blow-down along the infinite section, which is a $(-1)$-curve since $L\cong \mcO(1)$. By composing these morphisms,  we obtain a morphism
 \begin{equation}\label{mtop2}
\psi \colon \mcM\longrightarrow  \mbL \longrightarrow \mbP(\mcO\oplus L)\longrightarrow \mbP^2.
 \end{equation}
 In terms of coordinates, the composition $\mbL\rightarrow \mbP(\mcO\oplus L)\rightarrow \mbP^2$ is given by
\[
(q, p)\mapsto (q: p: 1) \quad \text{on $U_0$, and } 
(\tilde{q}, \tilde{p})\mapsto (1: \tilde{p}: \tilde{q}) \quad \text{on $U_\infty$}.
\]

\section{Moduli space of logarithmic connections and initial condition space of the difference Painlev\'e equation}
In this section, we show that the moduli space of $\boldsymbol{\alpha}$-stable parabolic logarithmic connections with spectral type $(2^2, 1^4, 1^4)$ and $(3^2, 2^3, 1^6)$ is isomorphic to the space of initial condition space of the difference Painlev\'e equation with symmetry of type $E^{(1)}_7$ and $E^{(1)}_8$, respectively.

\subsection{Normal form of logarithmic connections with spectral type $(2^2, 1^4, 1^4)$, $(3^2, 2^3, 1^6)$}\label{normalform}
By Proposition \ref{bdlE7}, \ref{bdlE8}, the underlying bundles of $\boldsymbol{\alpha}$-stable parabolic connections are isomorphic to $\mcO(1)\oplus \mcO^{\oplus r-1}$. We take a local basis $e_1, \ldots, e_r$ (resp. $\tilde{e}_1, \ldots, \tilde{e}_r$) of $\mcO(1)\oplus \mcO^{\oplus r-1}$ on $U_0$ (resp. on $U_\infty$) satisfying 
$e_i=\tilde{e}_i$ for $i=1, \ldots, r-1$ and $e_r=w\tilde{e}_r$. Then, $\nabla$ is written as $\nabla(e_r, \ldots, e_1)=(e_r, \ldots, e_1)A(z)\tfrac{dz}{z(z-1)}$ on $U_0$ and $\nabla(\tilde{e}_r, \ldots, \tilde{e}_1)=(\tilde{e}_r, \ldots, \tilde{e}_1)B(w)\frac{dw}{w(w-1)}$ on $U_\infty$. Our aim is to write down $A(z)$ and $B(w)$ by the parameters $q, p, \tilde{q}, \tilde{p}$. Let $\{F^i\}$ be the filtration of subseaves of $E$ satisfying Assumption \ref{appfil}. Let $E^i$ be the saturation of $F^i$. We assume that $e_{j+1}, \ldots, e_r$ generate $E^j$ for all $j$.

First, we consider $\boldsymbol{\nu}$-parabolic connections with spectral type $(2^2, 1^4, 1^4)$ for $\boldsymbol{\nu} \in \mcN$.

\begin{proposition}\label{AE7}
Assume that $\App(E, \nabla)\neq \infty$. Then, $A(z)$ is written as
\[A(z)=
\begin{pmatrix}
	0&a_{1\infty}z(z-1)+a_{11}z-a_{10}(z-1)&z(a_{2\infty}(z-1)+a_{21})&z(a_{3\infty}(z-1)+\mu_1)\\
	1&(\nu_0+\nu_1)z-\frac{1}{2}\nu_0&0&0\\
	0&z&-\frac{1}{2}\nu_0-p&-\mu(z-1) \\
	0&0&z-q&p
\end{pmatrix},
\]
where $a_{kl}, \mu_j$ are constants. $\mu_0, \mu_1$ satisfy with
\begin{equation}\label{muE7}
	q\mu_0=(p+\nu^{(0)}_{1})(p+\nu^{(0)}_{2}), \quad (q-1)\mu_1=(p-\nu^{(1)}_{1})(p-\nu^{(1)}_{2})(p-\nu^{(1)}_{3})(p-\nu^{(1)}_{4}), 
\end{equation}
and the constant $a_{kl}$ is uniquely determined by $q, p,\mu_0, \mu_1$, and $\boldsymbol{\nu}$. Under this form, we have $\App(E, \nabla)=q$ and $\varphi(E, \nabla, l_*)=(q, p\tfrac{dz}{z(z-1)})$.
\end{proposition}
\begin{proof}
Let $q=\App(E, \nabla)$. Since $T^{(0)}=k(q)$, $T^{(1)}=k(0)$ and $T^{(2)}=T^{(3)}=0$, 
$A(z)$ can be written as
\[A(z)=
\begin{pmatrix}
	a_{11}(z)&a_{12}(z)&a_{13}(z)&a_{14}(z)\\
	1&a_{22}(z)&a_{23}(z)&a_{24}(z)\\
	0&z&a_{33}(z)&a_{34}(z) \\
	0&0&z-q&a_{44}(z)
\end{pmatrix}, 
\]
where $a_{ii}(z), a_{jk}(z)$ are linear polynomials for all $i$ and $1<j<k$, and $a_{1k}(z)$ is a quadratic polynomial for $k>1$.
For a suitable automorphism of $E$, we may assume that $a_{33}(z)$ and $a_{44}(z)$ are constants, $a_{34}(z)$ is of the form $-\mu_0(z-1)$, and $a_{11}(z)=a_{23}(z)=a_{24}(z)=0$. $A(0)$ is a block upper triangular matrix. Since the eigenpolynomial of each diagonal block of $\res_{0}\nabla=-A(0)$ is $(\lambda-\nu^{(0)}_1)(\lambda-\nu^{(0)}_2)$, and $\tr A(z)=\nu_0(z-1)+\nu_1z$, the polynomial $a_{22}(z)$ and the constants $a_{12}(0)$ and $a_{33}(z)$ are uniquely determined by $a_{44}(z)$ and $\boldsymbol{\nu}$. Moreover, we can find the relation satisfied by $\mu_0$. From the condition $(\res_{0}\nabla)(l^{(0)}_i)\subset l^{0}_{i+1}$, we can see that $a_{13}(0)=a_{14}(0)=0$. Hence, $A(z)$ can be written as the assertion. We can find the relation satisfied by $a_{kl}$ and $\mu_{1}$ using the condition that the eigenpolynomial of $\res_{t_i}\nabla$ is $(\lambda-\nu^{(1)}_1)(\lambda-\nu^{(1)}_2)(\lambda-\nu^{(1)}_3)(\lambda-\nu^{(1)}_4)$.
\end{proof}
\begin{remark}\label{blowrem}
From (\ref{muE7}), when $q=0$, we see that $p=-\nu^{(0)}_i$ and $\mu_0$ is a free parameter. This motivates blowing up $\mbP^2$ at $(0:-\nu^{(0)}_i:1)$. By a similar argument, we are motivated to blow up $\mathbb{P}^2$ at $(1: \nu^{(1)}_i:1)$.
\end{remark}

Similarly, we have the following.
\begin{proposition}
Assume that $\App(E, \nabla, l_*)\neq 0$. Then, $B(w)$ is written as
\[B(w)=
\begin{pmatrix}
	0&b_{1\infty}w(w-1)+b_{11}w-b_{10}(w-1)&b_{2\infty}(w-1)+b_{21}w&\lambda_2(w-1)+\lambda_1w\\
	1&-(1+\frac{1}{2}\nu_0)w-\nu_\infty&0&0\\
	0&1&-\frac{1}{2}\nu_0w-\tilde{p}&-\nu^{(0)}_{0}\nu^{(0)}_{1}w\\
	0&0&w-\tilde{q}&\tilde{p}
\end{pmatrix}
\]
where $b_{kl}, \lambda_j$ are constants. $\lambda_1, \lambda_2$ satisfy with
\begin{align}
\begin{split}\label{lambdaE7}
(\tilde{q}-1)\lambda_{1}=&(\nu^{(0)}_{1}\nu^{(0)}_{2})^2(\tilde{q}-1)^2+(\tilde{p}-\nu^{(1)}_{1})(\tilde{p}-\nu^{(1)}_{2})(\tilde{p}-\nu^{(1)}_{3})(\tilde{p}-\nu^{(1)}_{4})\\
&+(2\tilde{p}^2-\nu_1\tilde{p}+\tfrac{1}{4}\nu_0^2+\tfrac{1}{2}\nu_0\nu_1+f_2(\nu^{(1)}))\nu^{(0)}_{1}\nu^{(0)}_{2}(\tilde{q}-1),\\
\tilde{q}\lambda_2=&(\tilde{p}+\nu^{(2)}_{1})(\tilde{p}+\nu^{(2)}_{2})(\tilde{p}+\nu^{(2)}_{3})(\tilde{p}+\nu^{(2)}_{4}),	
\end{split}
\end{align}
and the constants
 $b_{kl}$ is uniquely determined by $\tilde{q}, \tilde{p},\lambda_1, \lambda_2$, and $\boldsymbol{\nu}$. Under this form, we have $\App(E, \nabla, l_*)=\tilde{q}$ and $\varphi(E, \nabla, l_*)=(\tilde{q}, \tilde{p}\frac{dw}{w(w-1)})$.
\end{proposition}
In Appendix \ref{abform}, we give the explicit form of $a_{kl}$ and $b_{kl}$.

Next, we consider $\boldsymbol{\nu}$-parabolic connections with spectral type $(3^2, 2^3, 1^6)$ for $\boldsymbol{\nu} \in \mcN$. 
\begin{proposition}\label{AE8}
	Assume that $\App(E, \nabla)\neq \infty$. Then, $A(z)$ is written as
	\[A(z)=
	\begin{pmatrix}
		0&a_1(z)&a_2(z)&a_3z(z-1)&a_4z(z-1)&a_5z(z-1)\\
		1&(\nu_0+\nu_1)z-\frac{1}{3}\nu_0&0&0&0&0\\
		0&z&-\frac{2}{3}\nu_0-\frac{1}{2}\nu_1&-\mu_1(z-1)&0&0\\
		0&0&z-1&\frac{1}{3}\nu_0+\frac{1}{2}\nu_1&a_{7}z&a_{8} z\\
		0&0&0&z&-p-\frac{1}{3}\nu_0&-\mu_0(z-1)\\
		0&0&0&0&z-q&p
	\end{pmatrix},
	\]
	where $a_1(z)$ is a quadratic polynomial, $a_2(z)$ is a linear polynomial, $a_i, \mu_j$ are constants. $\mu_0, \mu_1$ satisfy with
	\begin{equation}\label{muE8}
	q\mu_0=(p+\nu^{(0)}_{1})(p+\nu^{(0)}_{2}), \quad
	(q-1)\mu_1=(p-\nu^{(1)}_{1})(p-\nu^{(1)}_{2})(p-\nu^{(2)}_{3}),
	\end{equation}
	and the polynomials $a_1(z), a_2(z)$ and the constant $a_{i}$ is uniquely determined by $q, p,\mu_0, \mu_1$, and $\boldsymbol{\nu}$. Under this form, we have $\App(E, \nabla, l_*)=q$ and $\varphi(E, \nabla, l_*)=(q, p\tfrac{dz}{z(z-1)})$.
\end{proposition}
\begin{proof}
Let $q=\App(E, \nabla)$. Since $T^{(0)}=k(q)$, $T^{(1)}=T^{(3)}=k(0)$, $T^{(2)}=k(1)$, and $T^{(4)}=T^{(5)}=0$, 
$A(z)$ can be written as
\[A(z)=
\begin{pmatrix}
	a_{11}(z)&a_{12}(z)&a_{13}(z)&a_{14}(z)&a_{15}(z)&a_{16}(z)\\
	1&a_{22}(z)&a_{23}(z)&a_{24}(z)&a_{25}(z)&a_{26}(z)\\
	0&z&a_{33}(z)&a_{34}(z)&a_{35}(z)&a_{36}(z)\\
	0&0&z-1&a_{44}(z)&a_{45}(z)&a_{46}(z)\\
	0&0&0&z&a_{55}(z)&a_{56}(z)\\
	0&0&0&0&z-q&a_{66}(z)
\end{pmatrix}.
\]
where $a_{ii}(z), a_{jk}(z)$ are linear polynomials for all $i$ and $1<j<k$, and $a_{1k}(z)$ is a quadratic polynomial for $k>1$. For a suitable automorphism of $E$, we may assume that $a_{ii}(z)$ are constants for $i>2$, $a_{3j}(1)=a_{56}(1)=0$ for $j>3$, $a_{4k}(0)=0$ for $k>4$, and $a_{11}(z)=a_{2l}(z)=0$ for $l>2$. From $(\res_{0}\nabla)(l^{(0)}_i)\subset l^{(0)}_{i+1}$ for all $i$, we can see that $a_{1j}(0)=a_{3k}(0)=0$ for $j>2$ and $k>4$, which implies that  $a_{3k}(z)=0$ for $k>4$. From $(\res_{1}\nabla)(l^{(1)}_i)\subset l^{(1)}_{i+1}$ for all $i$, we can see that $a_{1i}(1)=0$ for $i>3$. Hence, $A(z)$ can be written as the assertion. We can determine entries in the same manner as the proof of Proposition \ref{AE7}.
\end{proof}
\begin{proposition}
	Assume that $\App(E, \nabla, l_*)\neq 0$. Then, $B(w)$ is written as
	\[B(w)=
	\begin{pmatrix}
		0&b_1(w)&b_2(w)&b_3(w)&b_4(w)&b_{5}w-\lambda_2(w-1)\\\\
		1&-(\frac{\nu_0}{3}+1)w-\nu_2&0&0&0&0\\
		0&1&-(\frac{2}{3}\nu_0+\frac{\nu_1}{2})w&b_{6}w&b_{7}w&b_{8}w\\
		0&0&-w+1&(\frac{1}{3}\nu_0+\frac{1}{2}\nu_1)w&b_{9}w&\lambda_1w\\
		0&0&0&1&-\frac{\nu_0}{3}w-\tilde{p}&-\nu^{(0)}_0\nu^{(0)}_1w\\
		0&0&0&0&w-\tilde{q}&\tilde{p}
	\end{pmatrix}
	\]
	Here, $\lambda_1, \lambda_2$ satisfy with
	\begin{align}
	\begin{split}\label{lambdaE8}
		(\tilde{q}-1)\lambda_{1}=&\nu^{(0)}_{1}\nu^{(0)}_{2}(\tilde{q}-1)(\tilde{p}-\tfrac{1}{3}\nu_0-\nu_1)+(\tilde{p}-\nu^{(1)}_{1})(\tilde{p}-\nu^{(1)}_{2})(\tilde{p}-\nu^{(1)}_{3}),\\
		\tilde{q}\lambda_2=&(\tilde{p}+\nu^{(2)}_{1})(\tilde{p}+\nu^{(2)}_{2})(\tilde{p}+\nu^{(2)}_{3})(\tilde{p}+\nu^{(2)}_{4})(\tilde{p}+\nu^{(2)}_{5})(\tilde{p}+\nu^{(2)}_{6}),	
	\end{split}
	\end{align}
	and all scalar $b_{k}$ is uniquely determined by $\tilde{q}, \tilde{p},\lambda_1, \lambda_2$, and $\boldsymbol{\nu}$. Under this form, we have $\App(E, \nabla, l_*)=\tilde{q}$ and $\varphi(E, \nabla, l_*)=(\tilde{q}, \tilde{p}\tfrac{dw}{w(w-1)})$.
\end{proposition}
\begin{proof}
By the same argument as in the proof of Proposition \label{AE8}, $B(w)$ can be written as the assertion. $b_6$ and $b_9, \lambda_1$ are determined by the eigenpolynomials of the block matrices of $\res_0\nabla$ and $\res_1\nabla$, respectively. We can determine $b_7$ and $b_8$ as follows. The residue endomorphism $\res_0\nabla$ is written as 
\[
\res_0\nabla=
\begin{pmatrix}
B_{11}&B_{12}&B_{13}\\O&B_{22}&B_{23}\\O&O&B_{33}
\end{pmatrix},
\]
where $B_{ij}$ is a $2\times 2$ matrix. Let $^t(u_{i1}, \cdots , u_{i6})$ be a column eigenvector of $\res_0\nabla$ for $\nu^{(0)}_i$. Then, $^t(u_{i5}, u_{i6})$ is a column eigenvector of $B_{33}$ for $\nu^{(0)}_i$. Moreover, we have $B_{23}{}^t(u_{i5}, u_{i6})=(\nu^{(0)}_i-B_{22}){}^t(u_{i5}, u_{i6})$. Since 
\[
\nu^{(0)}_i-B_{22}=
\begin{pmatrix}
\frac{1}{3}\nu_0+\frac{1}{2}\nu_1-\nu^{(0)}_i\\1
\end{pmatrix}
\begin{pmatrix}
-1&\frac{1}{3}\nu_0+\frac{1}{2}\nu_1+\nu^{(0)}_i
\end{pmatrix},
\]
we have
\[
\begin{pmatrix}
b_7&b_8
\end{pmatrix}
=
\begin{pmatrix}
u_{15}&u_{25}\\u_{16}&u_{26}
\end{pmatrix}
\begin{pmatrix}
	\frac{1}{3}\nu_0+\frac{1}{2}\nu_1-\nu^{(0)}_1&0\\0&\frac{1}{3}\nu_0+\frac{1}{2}\nu_1-\nu^{(0)}_2
\end{pmatrix}
\begin{pmatrix}
	u_{15}&u_{25}\\u_{16}&u_{26}
\end{pmatrix}^{-1}.
\]
In the same manner, we can determine other coefficients.
\end{proof}

\subsection{Moduli spaces of logarithmic connections with spectral type $(2^2, 1^4, 1^4)$, $(3^2, 2^3, 1^6)$}\label{modsec}
In this section, we show that moduli spaces of parabolic logarithmic connection with spectral type $(2^2, 1^4, 1^4), (3^2, 2^3, 1^6)$ and degree 1 are isomorphic to the surfaces obtained by blowing up $S_{\boldsymbol{\nu}}$ of $\mbP^2$ at suitable points without an effective divisor $D_{\boldsymbol{\nu}}$. Set 
\[
t_0=0, \quad t_1=1, \quad t_2=\infty.
\]
Let $\mcS \rightarrow \mbP^2 \times \mcN$ be the blow-up along the $s_0+s_1+s_2$ sections $p^{(i)}_j\times \id \colon \mcN \rightarrow \mbP^2 \times \mcN$ defined by
\begin{equation}\label{pij}
p^{(i)}_j=
\left\{
\begin{array}{ll}
	(0: -\nu^{(0)}_j:1)& i=0, j=1, \ldots, s_0\\(1: \nu^{(1)}_j:1)&  i=1, j=1, \ldots, s_1, \\(1: -\nu^{(2)}_j:0)& i=2, j=1, \ldots, s_2.
\end{array}\right.	
\end{equation}
Let $\mcD \subset \mcS$ be the strict transform of $\{x_0(x_0-x_2)x_2=0\}\times \mcN$.
Let $S_{\boldsymbol{\nu}}$ and $D_{\boldsymbol{\nu}}$ be the fibers of $\mcS$ and $\mcD$ at $\boldsymbol{\nu} \in \mcN$.

\begin{theorem}\label{MT1} 
Assume that $0<\alpha^{(i)}_j\ll 1$ for all $i,j$. Then, the morphism $\psi\colon \mcM\rightarrow \mbP^2$ induces an isomorphism $\mcM\rightarrow \mcS\setminus \mcD$ over $\mcN$.
\end{theorem}
\begin{proof}
Since the map
\[
\mcM\setminus \cup_{i=0}^2\App^{-1}(t_i) \longrightarrow \mbL\setminus \cup_{i=0}^2\pi^{-1}(t_i)\times \mcN
\]
induced by $\varphi$ is an isomorphism by Proposition \ref{AE7}, \ref{AE8}, $\varphi$ is a birational morphism. We consider the fiber of $\App$ and $\pi$ at $t_i$. We have 
\begin{equation}\label{lsec}
\varphi(M^{\boldsymbol{\alpha}}(\boldsymbol{\nu})) \cap 
\pi^{-1}(t_i)=
\left\{
\begin{array}{ll}
	\left\{\left(0 (=q), -\nu^{(0)}_j\tfrac{dz}{z(z-1)}\right)\right\}_{j=1}^{s_0} &i=0\\
	\left\{\left(1, \nu^{(1)}_j\tfrac{dz}{z(z-1)}\right)\right\}_{j=1}^{s_1}&  i=1 \\
	\left\{\left(0 (=\tilde{q}), -\nu^{(2)}_j\tfrac{dw}{w(w-1)}\right)\right\}_{j=1}^{s_2}& i=2
\end{array}\right.
\end{equation}
by (\ref{muE7}), (\ref{lambdaE7}), (\ref{muE8}), and (\ref{lambdaE8}). Let $\widetilde{\mbL}$ be the blow up of $\mbL\times \mcN$ along sections defined by (\ref{lsec}) and $f_i$ be the strict transform of $\pi^{-1}(t_i)\times \mcN$. Then, from (\ref{muE7}), (\ref{lambdaE7}), (\ref{muE8}), (\ref{lambdaE8}), and Remark \ref{blowrem}, we can see that the morphism $\varphi$ lifts the morphism $\tilde{\varphi} \colon \mcM\rightarrow \widetilde{\mbL}$ and we have an isomorphism $\mcM\rightarrow \widetilde{\mbL}\setminus \left(f_0 \cup f_1 \cup f_2\right)$. Since the contraction $\mbL\rightarrow \mbP^2$ induces an isomorphism $\widetilde{\mbL}\setminus \left(f_0 \cup f_1 \cup f_2\right) \rightarrow \mcS\setminus \mcD$, we have the assertion.
\end{proof}
\subsection{The space of initial conditions of the difference Painlev\'e equation with symmetry of type $E^{(1)}_7$}\label{E7sec}
In this section, we consider the relation of $S_{\boldsymbol{\nu}}$ for $\bm{r}=(2^2, 1^4, 1^4)$ and the space of initial conditions of the difference Painlev\'e equation with symmetry of type $E^{(1)}_7$.

Let $D_0\subset S_{\boldsymbol{\nu}}$ be the strict transform of $\{x_0=0\}$. In the case where $\bm{r}=(2^2, 1^4, 1^4)$, the divisor $D_0$ is a $(-1)$-curve since $s_{0}=2$. We consider the contraction of $D_0$. The Picard group of $S_{\boldsymbol{\nu}}$ has the decomposition $\Pic(S_{\boldsymbol{\nu}})=\mbZ H\oplus \oplus_{i=0}^2\oplus_{j=1}^{s_i}\mbZ\mcE^{(i)}_j$, where $H$ is a line of $\mbP^2$ and $\mcE^{(i)}_j$ is the exceptional curve at $p^{(i)}_j$. If the morphism induced by the divisor $\mcE$ contracts $D_0$, then the intersection number $\mcE. D_0$ must be zero. Let $\mcE=2H-\mcE^{(0)}_1-\mcE^{(0)}_2$. Since $D_0=H-\mcE^{(0)}_1-\mcE^{(0)}_2$, we have $\mcE. D_0=0$. Global sections of $\mcO_{S_{\boldsymbol{\nu}}}(\mcE)$ correspond to quadratic polynomials vanishing at $p^{(0)}_1$  and $p^{(0)}_2$. Hence, the cohomology $H^0(\mcO_{S_{\boldsymbol{\nu}}}(\mcE))$ is generated by
\[
x_0^2, \ x_0(x_1+\nu^{(0)}_1x_2), \ x_0(x_1+\nu^{(0)}_2x_2), \ (x_1+\nu^{(0)}_1x_2)(x_1+\nu^{(0)}_2x_2).
\]
In particular, $\mcO_{S_{\boldsymbol{\nu}}}(\mcE)$ defines a morphism $S_{\boldsymbol{\nu}} \to \mbP^3$, which factors the morphism
\begin{equation}\label{sxmor}
S_{\boldsymbol{\nu}} \longrightarrow \mbP^1\times \mbP^1, \quad (x_0:x_1:x_2)\longmapsto \left((x_0:-(x_1+\nu^{(0)}_1x_2)), (x_0: x_1+\nu^{(0)}_2x_2)\right)
\end{equation}
and the Segre embedding $\mbP^1\times \mbP^1 \rightarrow \mbP^3$. 
Let $\mcX\rightarrow \mbP^1\times \mbP^1\times \mcN$ be the blow-up at the following sections $P_i \times \id \colon \mcN \rightarrow \pl\times \pl \times \mcN$;
\begin{equation}\label{E7pt}
P_i=\left\{
\begin{array}{ll}
	((1:-(\nu^{(1)}_{i}+\nu^{(0)}_1)), (1:\nu^{(1)}_{i}+\nu^{(0)}_2))& i=1, \ldots, 4 \\
	((1: \nu^{(2)}_{i}), (1:-\nu^{(2)}_{i}))& i=5, \ldots, 8. \\
\end{array}
\right.
\end{equation}
Let $\mcY_1, \mcY_2 \subset \mcX$ be the strict transforms of
\[
\left\{(y_1+\nu^{(0)}_0y_0)z_0+(z_1-\nu^{(0)}_1z_0)y_0=0\right\}, \quad \left\{y_1z_0+z_1y_0=0\right\}, 
\] 
on $\pl_{(y_0:y_1)} \times \pl_{(z_0:z_1)}\times \mcN$, respectively.
Then, the morphism (\ref{sxmor}) induces an isomorphism $\mcS\rightarrow \mcX$. In particular, we have an isomorphism $\mcS\setminus \mcD\rightarrow \mcX\setminus \left(\mcY_1\cup \mcY_2\right)$. The fiber $X_{\boldsymbol{\nu}}$ of $\mcX \rightarrow \mcN$ at $\boldsymbol{\nu}$ is the space of initial conditions for the difference Painlev\'e equations of type $E^{(1)}_7$ due to Kajiwara, Noumi, and Yamada \cite{KNY}, and $(\mcY_1)_{\boldsymbol{\nu}}+ (\mcY_2)_{\boldsymbol{\nu}}$ is the anti-canonical divisor of $X_{\boldsymbol{\nu}}$. Therefore, we obtain the following. 

\begin{theorem}\label{icsE7}
When $\bm{r}=(2^2, 1^4, 1^4)$, the family $\mcM$ of the moduli spaces is isomorphic to the family $\mcX\setminus \left(\mcY_1\cup \mcY_2\right)$ of the spaces of initial conditions for the difference Painlev\'e equations of type $E^{(1)}_7$ without the effective anti-canonical divisors.
\end{theorem}

\begin{figure}
	\centering\includegraphics[width=10cm]{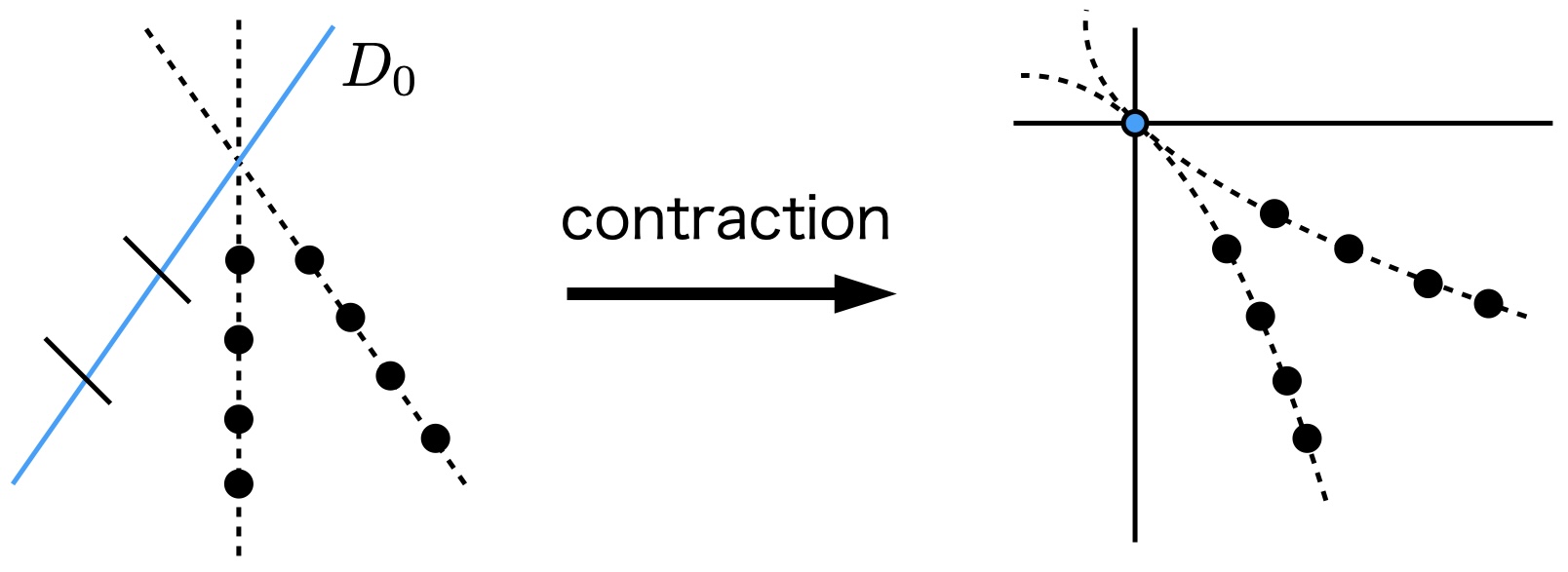}
\end{figure}

\subsection{The space of initial conditions of the difference Painlev\'e equation with symmetry of type $E^{(1)}_8$}\label{E8sec}
In this section, we consider the relation of $S_{\boldsymbol{\nu}}$ for $\bm{r}=(3^2, 2^3, 1^6)$ and the space of initial conditions of the difference Painlev\'e equation with symmetry of type $E^{(1)}_8$.

Since $D_0$ is a $(-1)$-curve, we obtain a smooth surface $Z_{\boldsymbol{\nu}}$ by contracting $D_0$. As in the $E^{(1)}_7$-case, the surface $Z_{\boldsymbol{\nu}}$ is obtained by blowing-up $\mbP^1\times \mbP^1$ at the following nine points;
\[
q_i=\left\{
\begin{array}{ll}
	((1:-(\nu^{(1)}_i+\nu^{(0)}_1)), (1:\nu^{(1)}_i+\nu^{(0)}_2))& i=1, 2, 3 \\
	((1: \nu^{(2)}_{i-3}), (1:-\nu^{(2)}_{i-3}))& i=4,\ldots, 9  
\end{array}
\right.
\]
The morphism $S_{\boldsymbol{\nu}}\rightarrow Z_{\boldsymbol{\nu}}$ is obtained by the rational map $\mbP^2\rightarrow \mbP^1_{(y_0:y_1)}\times \mbP^1_{(z_0:z_1)}$ defined by 
\[
(x_0:x_1:x_2)\longmapsto ((x_0:-(x_1+\nu^{(0)}_1x_2)), (x_0: x_1+\nu^{(0)}_2x_2)).
\]
Let $D_1 \subset S_{\boldsymbol{\nu}}$ be the strict transform of $\{x_0=x_2\} \subset \mbP^2$. Since $D_1$ is a $(-2)$-curve and intersects with $D_0$ at a point, the image $Y_1$ of $D_1$ by $S_{\boldsymbol{\nu}}\rightarrow Z_{\boldsymbol{\nu}}$ is a $(-1)$-curve. We consider the contraction of  $Y_1$. Let $H_1=\mbP^1\times \{\textrm{pt}\}, H_2=\{\textrm{pt}\} \times \mbP^1$ and $\mcF_i$ be the exceptional curve at $q_i$ for $i=1, \ldots, 9$. Then, the Picard group $\Pic (Z_{\boldsymbol{\nu}})$ has the decomposition $\Pic (Z_{\boldsymbol{\nu}})\cong \mbZ H_1\oplus \mbZ H_2 \oplus \bigoplus_{i=1}^9 \mbZ \mcF_i$. Let $\mcD=H_1+2H_2-\mcF_1-\mcF_2-\mcF_3$. Since $Y_1=H_1+H_2-\mcF_1-\mcF_2-\mcF_3$, we have $\mcD. Y_1=0$. Global sections of $\mcO_{Z_{\boldsymbol{\nu}}}(\mcD)$ correspond to $(2, 1
)$ polynomials vanishing at $q_1, q_2, q_3$. Hence, the cohomology $H^0(\mcO_{Z_{\boldsymbol{\nu}}}(\mcD))$ is generated by
\begin{align*}
h_1=&y_0(y_0z_1+y_1z_0+(\nu^{(0)}_1-\nu^{(0)}_2)y_0z_0)\\
h_2=&y_1(y_0z_1+y_1z_0+(\nu^{(0)}_1-\nu^{(0)}_2)y_0z_0)\\
h_3=&(\nu^{(0)}_1+\nu^{(1)}_1)(\nu^{(0)}_1+\nu^{(1)}_2)(\nu^{(0)}_1+\nu^{(1)}_3)y^2_0z_0-y^2_1z_1\\
&+\kappa y_1^2z_0+(3(\nu^{(0)}_2)^1+\nu_1\nu^{(0)}_1+\nu^{(1)}_1\nu^{(1)}_2+\nu^{(1)}_2\nu^{(1)}_3+\nu^{(1)}_3\nu^{(1)}_1)y_0y_1z_0,
\end{align*}
where $\kappa=\tfrac{1}{6}(4\nu^{(0)}_1+2\nu^{(0)}_2+\nu_1)$. 

Let $\mcX\rightarrow \mbP^2\times \mcN$ be the blow-up at the following sections $P_i \times \id \colon \mcN \rightarrow \mbP^2 \times \mcN$;
\[
P_i=\left\{
\begin{array}{ll}
\left(-3(\nu^{(0)}_1+\nu^{(1)}_{i})+3\kappa: -2: (-3(\nu^{(0)}_1+\nu^{(1)}_{i})+3\kappa)^3\right)& i=1, 2, 3 \\
\left(3\nu^{(2)}_{i-3}+3\kappa:-2: (3\nu^{(2)}_{i-3}+3\kappa)^3\right)& i=4,\ldots, 9  
\end{array}
\right.
\]
Then, the rational map $\mbP^1\times \mbP^1\rightarrow \mbP^2$ defined by
\begin{equation}\label{zxmor}
\left((y_0:y_1), (z_0:z_1)\right)\longmapsto (3\kappa h_1+3h_2:-2h_1: \kappa_1h_1+\kappa_2h_2+27(\nu^{(0)}_1-\nu^{(0)}_2)h_3),
\end{equation}
induces the morphism $Z_{\boldsymbol{\nu}}\rightarrow X_{\boldsymbol{\nu}}$. Here, 
\begin{align*}
\kappa_1=&\kappa^3-(\nu^{(0)}_1+\nu^{(1)}_1)(\nu^{(0)}_1+\nu^{(1)}_2)(\nu^{(0)}_1+\nu^{(1)}_3)\\
\kappa_2=&-\tfrac{1}{3}(6\kappa+\nu^{(0)}_1-\nu^{(0)}_2)(\nu^{(0)}_2-\nu^{(0)}_1)\\
&+\tfrac{1}{3}\left((\nu^{(1)}_1)^2+(\nu^{(1)}_2)^2+(\nu^{(1)}_3)^2-\nu^{(1)}_1\nu^{(1)}_2-\nu^{(1)}_2\nu^{(1)}_3-\nu^{(1)}_3\nu^{(1)}_1\right).
\end{align*} 
We note that the image $H_2-\mcF_i$ is $P_i$ for $i=1,2,3$. Let $\mcY\subset \mcX$ be the strict transform of the curve $\{w_1^2w_2=4w_0^3\}\subset \mbP^2_{(w_1: w_2: w_3)}\times \mcN$. Then, we obtain the isomorphism $\mcS\setminus \mcD\rightarrow \mcX\setminus \mcY$ by (\ref{sxmor}) and (\ref{zxmor}). The surface $X_{\boldsymbol{\nu}}$ is the space of initial conditions for the difference Painlev\'e equations of type $E^{(1)}_8$ due to Sakai \cite{Sa1}, and $\mcY_{\boldsymbol{\nu}}$ is the anti-canonical divisor of $X_{\boldsymbol{\nu}}$.  Therefore, we obtain the following. 

\begin{theorem}\label{icsE8}
When $\bm{r}=(3^2, 2^3, 1^6)$, the family $\mcM$ of the moduli spaces is isomorphic to the family $\mcX\setminus \mcY$ of the spaces of initial conditions for the difference Painlev\'e equations of type $E^{(1)}_8$ without the effective anti-canonical divisors.
\end{theorem}
\begin{figure}
	\centering\includegraphics[width=10cm]{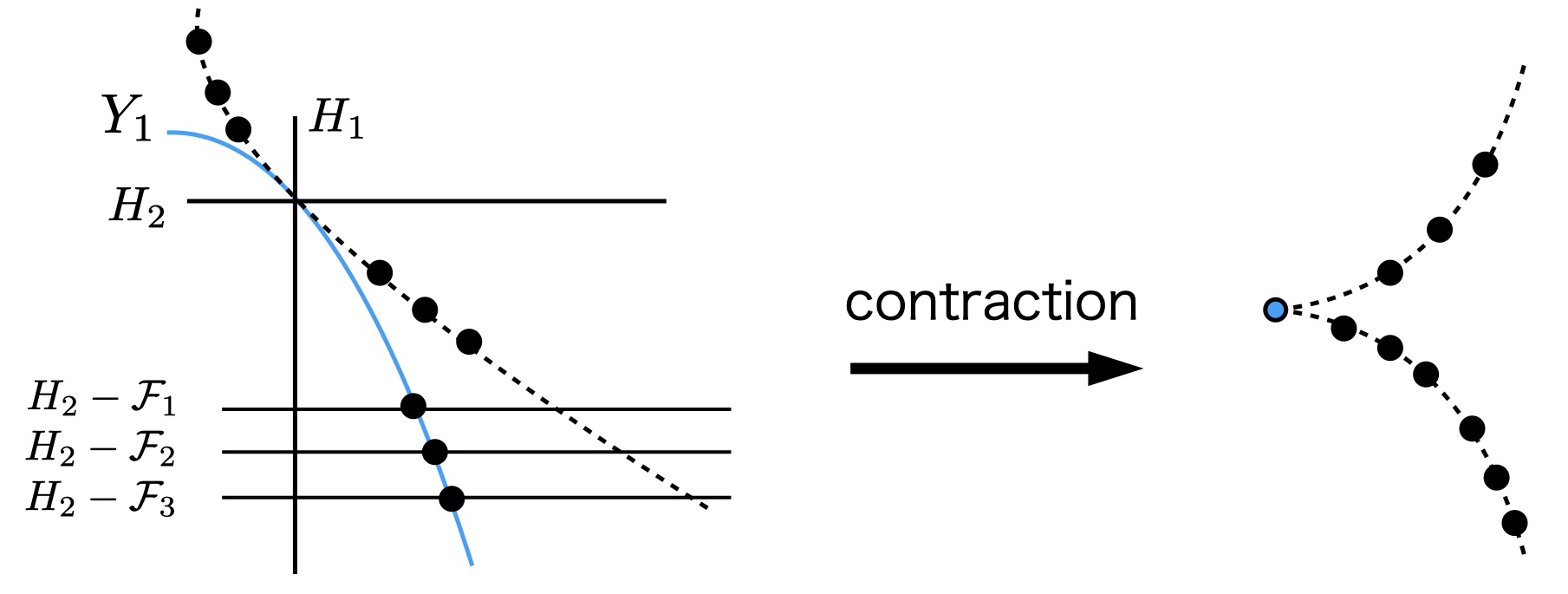}
\end{figure}

\section{Symmetry of the space of initial conditions}\label{symsec}
In this section, we explain how the extended affine Weyl group $\widetilde{W}(R^\perp)$ acts on the family $\mcS$ of $R$-surfaces.

\subsection{The extended affine Weyl group $\widetilde{W}(R^\perp)$}\label{E78}
The extended affine Weyl group $\widetilde{W}(R^\perp)$ is the semi-product  $\Aut(R^\perp) \ltimes W(R^\perp)$ of the automorphism group $\Aut(R^\perp)$ of the Dynkin diagram of type $R^\perp$ and the affine Weyl group $W(R^\perp)$. In the case of type $E^{(1)}_7, E^{(1)}_8$, $W(R^\perp)$ is defined as follows;
\[
W(R^\perp):=\left\langle  w_0, \ldots, w_l \vb
\begin{array}{cc}
	w_i^2=1& i=0, \cdots, l\\
	w_iw_j=w_jw_i&\text{$i$ and $j$ are not joined}\\
	w_iw_jw_i=w_jw_iw_j&\text{$i$ and $j$ are joined}
\end{array}
\right\rangle,
\]
where $i$ and $j$ mean the nodes of the Dynkin diagram.
\begin{equation}\label{dynkin}
\centering\includegraphics[width=10cm]{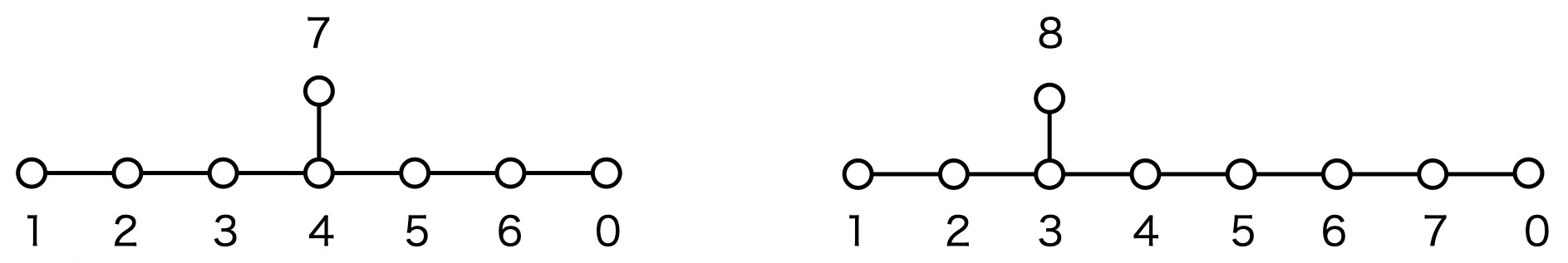}
\end{equation}
 The automorphism group $\Aut(E^{(1)}_7)$ of the Dynkin diagram of type $E^{(1)}_7$ is generated by $\sigma=(01)(26)(34)$. We note that $\Aut(A^{(1)}_2)$ is isomorphic to $\Aut(E^{(1)}_7)$. Since $\Aut(E^{(1)}_8)$ is trivial, the group $\widetilde{W}(E^{(1)}_8)$ coincides with $W(E^{(1)}_8)$.

\subsection{Action of $\widetilde{W}(R^\perp)$ on $\Pic (X_{\boldsymbol{\nu}})$}\label{actPicss}

In Sakai's theory, the Picard group $\Pic (X_{\boldsymbol{\nu}})$ has the root lattice of type $R^\perp$ as the sub-lattice, and the reflections define an action of $\tilde{W}=\widetilde{W}(R^\perp)$ on $\Pic (X)$. We see that in the case where $R^\perp=E^{(1)}_7$. 

The surface $X_{\boldsymbol{\nu}}$ is obtained by blowing up $\mbP^1 \times \mbP^1$ at the eight points $P_1, \ldots, P_8$ in (\ref{E7pt}). 
Let $H_1=\mbP^1\times \{\textrm{pt}\}, H_2=\{\textrm{pt}\} \times \mbP^1$ and $\mcF_i$ be the exceptional curve of $X$ at $P_i$ for $i=1, \ldots, 8$. Then, the Picard group $\Pic (X_{\boldsymbol{\nu}})$ has the decomposition $\Pic (X_{\boldsymbol{\nu}})\cong \mbZ H_1\oplus \mbZ H_2 \oplus \bigoplus_{i=1}^8 \mbZ \mcF_i$. Moreover, the intersection form is defined as follows;
\[
H_i^2=0, \quad \mcF_j^2=-1, \quad  H_1.H_2=1, \quad H_1. \mcF_j=H_2. \mcF_j=\mcF_i. \mcF_j=0\ (i \neq j).
\]
First, we consider the action of $W=W(E^{(1)}_7)$ on $\Pic (X_{\boldsymbol{\nu}})$. Set
\begin{equation}\begin{array}{llll}	
\alpha_0=\mcF_7-\mcF_8, & \alpha_1=\mcF_3-\mcF_4, & \alpha_2=\mcF_2-\mcF_3, & \alpha_3=\mcF_1-\mcF_2, \\
\alpha_4=H_2-\mcF_1-\mcF_5,& \alpha_5=\mcF_5-\mcF_6, & \alpha_6=\mcF_6-\mcF_7, & \alpha_7=H_1-H_2.
\end{array}
\end{equation}
The sublattice generated by $\alpha_0, \cdots, \alpha_7$ is isomorphic to the root lattice of type $E^{(1)}_7$. We denote it by $Q$, i.e., $Q:=\sum_{i=0}^7\mbZ \alpha_i$. The affine Weyl group $W$ is isomorphic to the group generated by the reflections $w_i$ of $\alpha_i$ on $Q$ for all $i$. The action of $W$ on $Q$ is naturally extended on $\Pic (X_{\boldsymbol{\nu}})$, that is, $W$ acts on $\Pic (X_{\boldsymbol{\nu}})$ by $w_i(C)=C+(C\cdot \alpha_i)\alpha_i$.

Next, we consider the action of $\Aut(E^{(1)}_7)$ on $\Pic (X_{\boldsymbol{\nu}})$. We denote the fibers $(\mcY_1)_{\boldsymbol{\nu}}$ and $(\mcY_2)_{\boldsymbol{\nu}}$ by $Y_1$ and $Y_2$, respectively. $Y_1, Y_2$ are written as elements in $\Pic (X_{\boldsymbol{\nu}})$ as follows;
\[
Y_{1}=H_1+H_2-\mcF_1-\mcF_2-\mcF_3-\mcF_4, \quad Y_{2}=H_1+H_2-\mcF_5-\mcF_6-\mcF_7-\mcF_8.
\]
The sub-lattice generated by $Y_1$ and $Y_2$ is isomorphic the root lattice of type $A^{(1)}_1$. We denote it by $Q^\perp$, i.e., $Q^\perp:=\mbZ Y_1+\mbZ Y_2$. We can see that $Q+Q^\perp=\Pic (X_{\boldsymbol{\nu}})$ and $Q \cap Q^\perp =\mbZ \delta$, where
\[
\delta=\alpha_1+2\alpha_2+3\alpha_3+4\alpha_4+3\alpha_5+2\alpha_6+\alpha_0+2\alpha_7=Y_{1}+Y_{2}.
\]
We set $\sigma(\alpha_j)=\alpha_{(01)(26)(35)(j)}$ (resp. $\sigma(Y_k)=Y_{(12)(k)}$). Then, $\sigma$ gives automorphisms of $Q$ and $Q^\perp$. This action extends to $Q+Q^\perp=\Pic (X_{\boldsymbol{\nu}})$, because it defines the same action on $Q \cap Q^\perp$. In fact, we have
\[
\sigma(\alpha_1)+2\sigma(\alpha_2)+3\sigma(\alpha_3)+4\sigma(\alpha_4)+3\sigma(\alpha_5)+2\sigma(\alpha_6)+\sigma(\alpha_0)+2\sigma(\alpha_7)=\sigma(Y_1)+\sigma(Y_2).
\]
 Therefore, $\sigma$ provides an automorphism of $\Pic (X_{\boldsymbol{\nu}})$. We find that these automorphisms define an action of $\Aut(E^{(1)}_7)$ on $\Pic (X_{\boldsymbol{\nu}})$ by the construction. In Appendix \ref{actpicX}, we give variations of the basis $H_1, H_2, \mcF_1, \ldots, \mcF_8$ by generators of $\widetilde{W}$. 

\subsection{Action of $\widetilde{W}(R^\perp)$ on $\Pic (S_{\boldsymbol{\nu}})$}
Since $S_{\boldsymbol{\nu}}$ is obtained by blowing up $X_{\boldsymbol{\nu}}$, the Picard group $\Pic (S_{\boldsymbol{\nu}})$ can be decomposed into $\Pic(X_{\boldsymbol{\nu}})$ and the group generated by the exceptional curves. In fact, we have $\Pic (S_{\boldsymbol{\nu}})=\Pic (X_{\boldsymbol{\nu}})\oplus L$, where
\[
L=
\left\{
\begin{array}{ll}
\mbZ D_0 &R^\perp=E^{(1)}_7, \\
\mbZ D_0 \oplus \mbZ D_1 &R^\perp=E^{(1)}_8. 
\end{array}\right.
\]
We define the action of $\widetilde{W}$ on $\Pic(S_{\boldsymbol{\nu}})$ by 
\[
w(B+C):=w(B)+C \quad (w \in \widetilde{W}, \ B \in \Pic(X_{\boldsymbol{\nu}}), \ C \in L).
\]
By the way, $\Pic (S_{\boldsymbol{\nu}})$ has the decomposition $\Pic (S_{\boldsymbol{\nu}})=\mbZ H \oplus \oplus_{i=0}^2\oplus_{j=1}^{s_i}\mbZ \mcE^{(i)}_j$. The relations of the bases of $\Pic (X_{\boldsymbol{\nu}})$ and $\Pic (S_{\boldsymbol{\nu}})$ is in Appendix \ref{picbasis}, and the explicit actions on $\Pic (S_{\boldsymbol{\nu}})$ is written in Appendix \ref{actpicS}.
\subsection{Action of $\widetilde{W}$ on $\mcN$}

The action of $\widetilde{W}$ on $\mcN$ comes from period integrals. We see that in the case where $R^\perp=E^{(1)}_7$.

 We consider a meromorphic two-form on $X_{\boldsymbol{\nu}}$ with a simple pole along the effective anti-canonical divisor $Y$. Since $H^0(\Omega_{X}^2(Y)) \cong H^0(\mcO_{X_{\boldsymbol{\nu}}})\cong \mathbb{C}$,  there is a unique meromorphic two-form $\omega$ in  $H^0(\Omega_{X_{\boldsymbol{\nu}}}^2(Y))$ with $\int_{\delta} \omega=1$. We define $\chi\colon Q\rightarrow \mathbb{C}$ by $\chi(C)=\int_{C}\omega$. By Lemma 21 in \cite{Sa1}, we have
\begin{equation}\label{int}
	\begin{array}{llll}
		\chi(\alpha_1)=\nu^{(1)}_4-\nu^{(1)}_3,  &\chi(\alpha_2)=\nu^{(1)}_3-\nu^{(1)}_2,  &\chi(\alpha_3)=\nu^{(1)}_2-\nu^{(1)}_1,&\chi(\alpha_4)=\gamma,  \\
		\chi(\alpha_5)=\nu^{(2)}_2-\nu^{(2)}_1,  &\chi(\alpha_6)=\nu^{(2)}_3-\nu^{(2)}_2, & \chi(\alpha_7)=\nu^{(0)}_2-\nu^{(0)}_1, &\chi(\alpha_0)=\nu^{(2)}_4-\nu^{(2)}_3,
	\end{array}	
\end{equation}
where $\gamma=\nu^{(0)}_1+\nu^{(1)}_1+\nu^{(2)}_1$. In particular, $\boldsymbol{\nu}$ is uniquely determined by the values $\chi(\alpha_0), \ldots, \chi(\alpha_7)$ and the condition
\begin{equation}\label{sumcond}
2\nu^{(0)}_{1}+2\nu^{(0)}_{2}=\nu_0, \quad \nu^{(i)}_{1}+\nu^{(i)}_{2}+\nu^{(i)}_{3}+\nu^{(i)}_{4}=\nu_i\ (i=1, 2).
\end{equation}
We define an action of $\widetilde{W}$ on $\mcN$ using this result, that is, we define $w(\nu^{(i)}_{j})$ as the solution of systems of linear equations obtained by replacing $\alpha_k$ and $\nu^{(i)}_{j}$ with $w(\alpha_k)$ and $w(\nu^{(i)}_{j})$ in the conditions (\ref{int}) and (\ref{sumcond}). 

We consider the case where $w=w_4$ as an example. Focusing on the first equation of (\ref{int}), we have
\[
w_4(\nu^{(1)}_4)-w_4(\nu^{(1)}_3)=\chi(w_4(\alpha_1))=\chi(\alpha_1)=\nu^{(1)}_4-\nu^{(1)}_3.
\]
The following equation is obtained by repeating the same calculation. 
\[\left\{
\begin{array}{ll}
	w_4(\nu^{(i)}_2)-w_4(\nu^{(i)}_1)=\nu^{(i)}_2-\nu^{(i)}_1+\gamma& i=0, 1, 2\\
	w_4(\nu^{(i)}_{j+1})-w_4(\nu^{(i)}_{j})=\nu^{(i)}_{j+1}-\nu^{(i)}_{j}& i=1, 2, \ j=2, 3\\
	w_4(\nu^{(0)}_1)+w_4(\nu^{(1)}_1)+w_4(\nu^{(2)}_1)=-\gamma\\
	2w_4(\nu^{(0)}_1)+2w_4(\nu^{(0)}_2)=\nu_0&\\
	w_4(\nu^{(i)}_1)+w_4(\nu^{(i)}_2)+w_4(\nu^{(i)}_3)+w_4(\nu^{(i)}_4)=\nu_i &i=1, 2.
\end{array}\right.
\]
The solution is as follows;
\begin{equation}
w_4(\nu^{(i)}_j)=\left\{
\begin{array}{cc}
	\nu^{(i)}_{j}-\frac{s_i-1}{s_i}\gamma & j=1\\
	\nu^{(i)}_{j}+\frac{1}{s_i}\gamma &  j\neq 1
\end{array}\right.
\end{equation}

We give the explicit action of the generators of $\widetilde{W}$ on $\mcN$ in Appendix \ref{actionN0}.

\subsection{Action of $\widetilde{W}$ on $\mcX$}
The action of $\widetilde{W}$ on $\mcX$ is defined by the automorphisms which realize the action on the Picard group.

We take a decomposition of the Picard groups as in the section \ref{actPicss}. 
\begin{equation}
\begin{array}{l}
\Pic (X_{\boldsymbol{\nu}})=\mbZ H_1 \oplus \mbZ H_2 \oplus\bigoplus_{i=1}^8\mbZ \mcF_i \\
\Pic (X_{w(\boldsymbol{\nu})})=\mbZ H'_1 \oplus \mbZ H'_2 \oplus\bigoplus_{i=1}^8\mbZ \mcF'_i
\end{array}
\end{equation}
Let us fix $w \in \widetilde{W}$. 
\begin{proposition}(Theorem 25 in \cite{Sa1})\label{actsurf}
	There is a unique isomorphism $\phi_w \colon X_{\boldsymbol{\nu}}\rightarrow X_{w(\boldsymbol{\nu})}$ such that
	\begin{equation}\label{actS}
		\phi_w^*(H'_i)=w(H_i), \quad \phi_w^*(\mcF'_j)=w(\mcF_j)
	\end{equation}
	for any $i=1, 2$ and $j=1, \ldots, 8$.
\end{proposition}

\subsection{Action of $\widetilde{W}$ on $\mcS$}\label{actS}
Since $S_{\boldsymbol{\nu}}$ is obtained by blowing up $X_{\boldsymbol{\nu}}$ at points, $\phi_w$ in Proposition \ref{actsurf} lifts as an isomorphism $S_{\boldsymbol{\nu}} \to S_{w(\boldsymbol{\nu})}$. We write this map by the same character $\phi_w$. In this section, we find the explicit form of $\phi_w$ when $w$ is a generator of $\widetilde{W}$. 

First, we consider the case where $w=w_i$ when $i$ is not the central nodes. In the case where $R^\perp=E^{(1)}_7$, the action of $w_1$ on $\Pic(S_{\boldsymbol{\nu}})$ and $\mcN$ is given by
\[
w_1(\mcE_i)=\mcE_{(34)(i)}
\]
and
\[
w_1(\nu^{(i)}_{j})=
\left\{
\begin{array}{cc}
	\nu^{(i)}_{(34)(j)} & i=1\\
	\nu^{(i)}_{j} & i\neq 1.
\end{array}
\right. 
\]
Hence, the surfaces $S_{\boldsymbol{\nu}}$ and $S_{w_1(\boldsymbol{\nu})}$ are also the same. We can check that 
$\phi_{w_1}:=\id \colon S_{\boldsymbol{\nu}}\rightarrow S_{w_1(\boldsymbol{\nu})}$ satisfies the condition (\ref{actS}). We can also see that $\phi_{w_i}:=\id \colon S_{\boldsymbol{\nu}}\rightarrow S_{w_i(\boldsymbol{\nu})}$ is the desired isomorphism for other cases.

Next, we consider the case where $w$ is the auromotphism $\sigma$ of the Dynkin diagram of type $E^{(1)}_7$. The action of $\sigma$ on $\Pic(S_{\boldsymbol{\nu}})$ is given by
\[
\sigma(\mcE^{(i)}_j)=\mcE^{(12)(i)}_j
\]
and we have $\sigma(D_j)=D_{(12)(j)}$. It means that $\phi_{\sigma}$ should exchange lines $\{x_0=x_2\}, \{x_x=0\}$ and moves the exceptional curve $\mcE^{(i)}_j$ to $\mcE^{(12)(i)}_j$. 
Let $\phi_{\sigma}\colon S_{\boldsymbol{\nu}} \rightarrow S_{\sigma(\boldsymbol{\nu})}$ be the isomorphisms induced by the isomorphism
\[
\mathbb{P}^2 \rightarrow \mathbb{P}^2; \  (x_0:x_1:x_2)\mapsto (x_0: \tfrac{1}{4}(\nu_1-\nu_2)x_0-x_1: x_0-x_2).
\]
We can see that $\phi_{\sigma}$ satisfies the condition (\ref{actS}).

Finally, we consider the case where $w$ is the reflection associated with the central node. The action of $w$ on $\Pic(S_{\boldsymbol{\nu}})$ is given by
\[w(H)=2H-\mcE^{(0)}_1-\mcE^{(1)}_1-\mcE^{(2)}_1, \quad 
w(\mcE^{(i)}_j)=
\left\{
\begin{array}{ll}
	H-\mcE^{(1)}_1-\mcE^{(2)}_1 &i=0, j=1\\
	H-\mcE^{(0)}_1-\mcE^{(2)}_1&i=1, j=1\\
	H-\mcE^{(0)}_1-\mcE^{(1)}_1&i=2, j=1\\
	\mcE^{(i)}_j & j\neq 1
\end{array}
\right.
\]
It implies that $\phi_{w}$ is the morphism induced by the quadratic transformation of $\mathbb{P}^2$ based at the points $p^{(0)}_1, p^{(1)}_1, p^{(2)}_1$ if those points do not lie on the same line, that is, $\gamma\neq 0$. Assume that $\gamma\neq 0$.
Let  $\phi_{w} \colon S_{\boldsymbol{\nu}}\rightarrow S_{w(\boldsymbol{\nu})}$ be the isomorphism induced by the quadratic transformation
\begin{equation}\label{quadtranf}
	\mathbb{P}^2 \dasharrow \mathbb{P}^2, \  (x_0:x_1:x_2)\mapsto (\gamma x_0f_{12}: af_{01}f_{12}-bf_{12}f_{02}-cf_{01}f_{02} :\gamma x_2f_{01}), 
\end{equation}
where
\[
a:=\nu^{(1)}_{1}-\tfrac{3}{4}\gamma, \ b:=\nu^{(2)}_{1}-\tfrac{3}{4}\gamma, \  c:=\nu^{(0)}_{1}-\tfrac{1}{2}\gamma,
\]
\begin{align*}
	&f_{01}(x_0, x_1, x_2):= (\nu^{(0)}_{1}+\nu^{(1)}_{1})x_0-x_1-\nu^{(0)}_{1}x_2,\\
	&f_{12}(x_0, x_1, x_2):=  -\nu^{(2)}_{1}x_0-x_1+(\nu^{(1)}_{1}+\nu^{(2)}_{1})x_2,\\
	&f_{02}(x_0, x_1, x_2):= \nu^{(2)}_{1}x_0+x_1+\nu^{(0)}_{1}x_2.
\end{align*}
We note that the zeros of $f_{ij}$ forms the line passing through $p^{(i)}_1, p^{(j)}_1$. When $\gamma \neq 0$, the map (\ref{quadtranf}) becomes a projective linear transformation. We can check that $\phi_{w}$ satisfies the condition (\ref{actS}). 

Isomorphisms in Proposition \ref{actsurf} are extended to automorphisms of $\mcS$. Those automorphisms define an action of $\widetilde{W}$ on $\mcS$. The action is equivariant for the map $\mcS\rightarrow \mcN$ by the construction. In the next section, we will see realizations of the action over $\mcS\rightarrow \mcN$ in terms of parabolic connections.

\section{Realization of symmetry of the space of initial conditions}\label{symsec}

Transformations of parabolic connections induce automorphisms of $\mcM$ by the universal property of the moduli spaces.
Our goal is to find transformations for any $w\in \widetilde{W}(R^\perp)$ such that the induced morphism $f_w\colon \mcM\rightarrow \mcM$ makes the following diagrams commutative
\begin{equation}\label{diag}
	\begin{tikzcd}
		\mcM\arrow[r, "f_w"]\arrow[d]&\mcM\arrow[d]\\
		\mcN\arrow[r, "w"]&\mcN,
	\end{tikzcd}  \quad 
	\begin{tikzcd}
		\mcM\arrow[r, "f_w"]\arrow[d, "\tilde{\psi}"]&\mcM\arrow[d, "\tilde{\psi}"]\\
		\mcS\arrow[r, "\phi_w"]&\mcS.
	\end{tikzcd}
\end{equation}
We will see such a realization in three cases;
\begin{itemize}
\item[(i)] the reflections associated with the non-central nodes
\item[(ii)]the automorphism of the Dynkin diagram
\item[(iii)] the reflection associated with the central node
\end{itemize}

\subsection{Realization of the reflection associated with the non-central node}

The action of the reflection $w_i$ associated with the non-central node on $\mcN$ is to exchange adjacent local exponents. Hence, the action of $w_i$ on $\mcM$ is realized by changing not the underlying logarithmic connection but the parabolic structures.
We will see it in the cases where $R^\perp=E^{(1)}_7$ and $w_0$. We can realize the action of other $w_i$ in the same manner. $w_0$ acts on $\mcN$ by
\[
w_0(\nu^{(i)}_j)=
\left\{
\begin{array}{ll}
	\nu^{(i)}_{(12)(j)}&i=0, \\\nu^{(i)}_j&i\neq 0
\end{array}\right.
\]
Hence, we only have to replace $l^{(0)}_{1}$ with the eigenspace of $\res_{0}\nabla$ for $\nu^{(0)}_{1}$.
That is, for a $\boldsymbol{\nu}$-parabolic logarithmic connection $(E,\nabla, l_*)$, we define a $w_0(\boldsymbol{\nu})$-parabolic logarithmic connection $(E',\nabla', l'_*)$ by $E':=E$, $\nabla':=\nabla$, $l'^{(i)}_{j}:=l^{(i)}_{j}$ for $(i,j)\neq(1,2)$, and 
\[
l'^{(0)}_{1}:=\Ker(\res_0\nabla-\nu^{(0)}_{1}\id\colon E|_{0}\longrightarrow E|_{1})
\]
The transformation $(E,\nabla, l_*) \mapsto (E',\nabla', l'_*)$ induces an automorphism $f_{w_1}\colon \mcM\rightarrow \mcM
$, and $f_{w_1}$ commutes the diagrams (\ref{diag}). 

We summarize realizations of other reflections in Appendix \ref{apprealize}.

\subsection{Realization of the automorphism}\label{realizeaut}
The action of $\sigma$ on $\mcN$ is as follows;
\[
\sigma(\nu^{(i)}_{j})=
\left\{
\begin{array}{cc}
\nu^{(0)}_{j} & i=0\\
\nu^{(2)}_{j}+\frac{1}{4}(\nu_1-\nu_2) & i=1\\
\nu^{(1)}_{j}+\frac{1}{4}(\nu_2-\nu_1) & i=2,
\end{array}
\right. 
\]
Hence, $\sigma$ is realized by exchanging local exponents at $z=1$ and $z=\infty$ and by shifting them. We can show the following in a same manner of Section 4.2 in \cite{Ma2}.  
\begin{proposition}
Given a $\boldsymbol{\nu}$-logarithmic connection $(E, \nabla, l_*)$, set 
\[
(E', \nabla'):=s^*(E, \nabla)\otimes (\mcO_{\pl}, d+\tfrac{\nu_1-\nu_2}{4}\tfrac{dz}{z-1}),
\]
where $s\colon \pl \rightarrow \pl$ is a morphism defined by $z\mapsto \frac{1}{z-1}$. 
Then, there is a unique parabolic structure $l'^{(i)}_{*}$ of $E'|_{i}$ for each $i$ such that the collection $(E', \nabla', l'_*:=\{l'^{(i)}_{*}\}_{i=0, 1, 2})$ is a $\sigma(\boldsymbol{\nu})$-parabolic connection. Let us define a morphism $f_{\sigma}\colon \mcM\rightarrow \mcM$ by $f_{\sigma}((E, \nabla, l_*)):=(E', \nabla', l'_*)$. Then $f_{\sigma}$ commutes the diagrams (\ref{diag}). 
\end{proposition}

\subsection{Simpson's middle convolution}\label{mcsec}
We review the theory of middle convolution of logarithmic connections over $(\pl, D=[t_1]+\cdots+[t_n])$ defined in \cite{Sim}. 
Put $(Y, D_Y):=(\pl, D)$ and $(Z, D_Z):=(\pl, D)$, and let  $p_1\colon Y\times Z\rightarrow Y$ and $p_2\colon Y\times Z\rightarrow Z$ be the projections. Let $\pi \colon X \rightarrow Y\times Z$ be the blow-up at the $n$ points $(t_1, t_1), \ldots, (t_n, t_n)$, and set $\pi_1=p_1 \circ \pi$ and $\pi_2=p_2 \circ \pi$. We define a normal crossing divisor $J$ on $X$ as
\[
J:=\sum_{i=1}^nH_i+\sum_{i=1}^nV_i+\sum_{i=1}^nU_i+T,
\]
where $H_i$, $V_i$, $T$ are the strict transformations of $\{t_i\}\times Z$, $Y \times \{t_i\}$, the diagonal curve, respectively, and $U_i$ is the exceptional curve corresponding to $(t_i, t_i)$. Set $\xjzq{1}:=\xj{1}/\pi_2^*\zd$.

We consider the middle convolution $\mc_{\beta}(E, \nabla)$ of logarithmic connections $(E, \nabla)$ over $(\pl, \bm{t}=\{t_i\}_{i=1}^n)$ for $\beta \in H^0(X, \xj{1})$. We define a logarithmic connection $(F, \nabla_F)$ over $(X,J)$ as
\[
(F, \nabla_F)=\pi_1^*(E, \nabla)\otimes (\mcO_X, d+\beta). 
\]
Consider the de Rham complex $\DR_X:=\DR_X(F, \nabla_F)$ and the relative de Rham complex $\rDR:=\DR_{X/Z}(F, \nabla_F)$ defined as follows;
\begin{align*}
	\DR_X(F, \nabla_F) :=&[F \overset{\nabla_F}{\longrightarrow}F\otimes \xj{1} \overset{\nabla_F}{\longrightarrow}F\otimes \xj{2}],\\
	\DR_{X/Z}(F, \nabla_F) :=&[F \overset{\overline{\nabla}_F}{\longrightarrow}F\otimes \Omega^1_{(X,J)/(Z,D)}],
\end{align*}
where $\overline{\nabla}_F$ is the composition of $\nabla_F$ and the quotient map $F\otimes\xj{1}\rightarrow F\otimes\xjzq{1}$. Two complexes $\DR_X$ and $\rDR$ fit into the short exact sequence 
\begin{equation}\label{drex}
	0\to \rDR[-1]\otimes \pi_2^*\zd\to \DR_X\to \rDR\to 0.
\end{equation}
When $\mathbb{R}^1\pi_{2*}\rDR$ is locally free, the connecting homomorphism of the above short exact sequence defines a connection map
\[
\nabla^{\GM}: \mathbb{R}^1\pi_{2*}\rDR \to \mathbb{R}^2\pi_{2*}(\rDR[-1]\otimes \pi_2^*\zd) \cong \mathbb{R}^1\pi_{2*}\rDR \otimes \zd,
\]
which is called a logarithmic Gauss-Manin connection. We define the middle convolution $\mc_{\beta}(E,\nabla)$ as a subconnection of $(\mathbb{R}^1\pi_{2*}\rDR, \nabla^{\GM})$. Set
\[
\mcF^0:=(F|_{T}/\Image (\res_T\nabla_F)) \oplus \bigoplus_{i=1}^n(F|_{H_i}/\Image (\res_{H_i}\nabla_F)),
\]
where $\res_T\nabla_F\colon F|_{T}\rightarrow F|_{T}$ and $\res_{H_i}\nabla_{F}\colon F|_{H_i}\rightarrow F|_{H_i}$ be the endmorphisms induced by $\nabla_F$. We define a sheaf $F'$ by the following short exact sequence 
\begin{equation}\label{fprime}
	0 \to F' \to F\to \mcF^0\to 0. 
\end{equation}
Then, we have $\overline{\nabla}_F(F)\subset F'\otimes \xjzq{1}$. We define a complex $\rMDR:=\MDR_{X/T}(F, \nabla_F)$ of sheaves over $\mbC$ by
\begin{equation}
	\MDR_{X/T}(F, \nabla_F) :=[F \overset{\overline{\nabla}_F}{\longrightarrow}F'\otimes \Omega^1_{(X,J)/(Z,D)}].
\end{equation}
There is a short exact sequence
\[
0 \to \rMDR\to \rDR \to \mcF^0[-1] \to 0, 
\]
which induces 
\[\pi_{2*} \mcF^0[-1]\to \mbR^1\pi_{2*}\rMDR \to \mbR^1\pi_{2*}\rDR.
\]
Since $\pi_{2*} \mcF^0[-1]=0$, the sheaf $\mbR^1\pi_{2*}\rMDR$  becomes a subsheaf of $\mbR^1\pi_{2*}\rDR$. By definition, we have 
\[
\nabla^{\GM}(\mbR^1\pi_{2*}\rMDR)\subset \mbR^1\pi_{2*}\rMDR\otimes \xjzq{1}.
\]
Let  $\nabla^{\GM}_{\tmid}$ be the restriction of $\nabla^{\GM}$ on $\mbR^1\pi_{2*}\MDR_{X/T}$. 

\begin{definition}
	Assume that $\mbR^1\pi_{2*}\rMDR$ is locally free. We call the pair $(\mbR^1\pi_{2*}\rMDR, \nabla^{\GM}_{\tmid})$ \textit{the middle convolution} of $(E, \nabla)$ for $\beta$, and denote it by $\mc_\beta(E,\nabla)$; 
	\[
	\mc_\beta(E,\nabla):=(\mbR^1\pi_{2*}\rMDR, \nabla^{\GM}_{\tmid}).
	\]
\end{definition}

Next, we will see the variation of local exponents. Let us fix $\boldsymbol{\mu}=(\mu^{(i)}_{j})_{1\leq j\leq r}^{1\leq i\leq n}\in \mbC^{rn}$. For $\mu \in \mbC$, we define a divisor $[\boldsymbol{\mu}]_i$ of $\mathbb{C}$ by $[\boldsymbol{\mu}]_i:=\sum_{\mu}m_i(\mu)[\mu]$, where $m_i(\mu):=\#\{j\mid \mu^{(i)}_{j}=\mu\}$. We set
\[
\delta(\beta, \boldsymbol{\mu}):=(n-2)r-\sum_{i=1}^{n}m_i(-\beta^{H_i}).
\]
\begin{proposition}(Scholium 5.8. in \cite{Sim})\label{variation}
	Given a $\boldsymbol{\mu}$-parabolic connection $(E, \nabla, l_*)$, set $(E', \nabla'):=\mc_{\beta}(E,\nabla)$. Let $\boldsymbol{\mu}'$ be the set of local exponents of $\nabla'$. Suppose that $\mu^{(i)}_{1}, \ldots, \mu^{(i)}_{r}$ don't differ by nonzero integers, $\mu+\beta^{H_i}+\beta^T\notin \mbZ$ and $\mu+\beta^{H_i}\notin \mbZ\setminus\{0\}$ for any local exponent of $\nabla'$ at $t_i$, and $\beta^T\notin \mbZ$. Then the rank of $E'$ is $r+\delta(\beta, \boldsymbol{\mu})$ and we have
	\begin{align*}
		[\boldsymbol{\mu}']_i=
		&(m_i(-\beta^{H_i})+\delta(\beta, \boldsymbol{\mu}))[\beta^{V_i}]+\sum_{\mu+\beta^{H_i}\neq 0}m_i(\mu)[\mu+\beta^{U_i}].
	\end{align*}
\end{proposition}
We denote the residue of $\beta\in H^0(X, \Omega_X^1(\log J))$ along $H_i, V_i, U_i, T$ by $\beta^{H_i}, \beta^{V_i}, \beta^{U_i}, \beta^{T}$. We can find by the residue formula for suitable curves that residues satisfy the relations
\begin{equation}\label{betarel}
	\beta^T+\sum_{i=1}^n\beta^{H_i}=0, \quad  \beta^T+\sum_{i=1}^n\beta^{V_i}=0, \quad \beta^{U_i}=\beta^{H_i}+\beta^{V_i}+\beta^{T}.
\end{equation}
Conversely, given $\beta^{H_i}, \beta^{V_i}, \beta^{U_i}, \beta^{T} \in \mathbb{C}$ satisfying the relations $(\ref{betarel})$, there is a unique $\beta \in H^0(X, \Omega_X^1(\log J))$ with these residues. 
\subsection{Realization of associated with the central node}\label{realizecn}
We finally consider the realization $w$ of the central node. We use middle convolution for logarithmic connections for realization. 

Let $\beta \in H^0(X, \Omega_X^1(\log J))$ with $\beta^T=\nu^{(0)}_{1}+\nu^{(1)}_{1}+\nu^{(2)}_{1}$, $\beta^{H_i}=-\nu^{(i)}_{1}$, $\beta^{V_i}=\nu^{(i)}_{1}-\tfrac{s_i-1}{s_i}\beta^T$ for $i=0, 1, 2$. Set
\[
\mcN^0:=\left\{\boldsymbol{\nu}=(\nu^{(i)}_{j})_{1\leq j\leq s_i}^{0\leq i\leq 2} \in \mcN \vb \text{$\nu^{(i)}_{j}-\nu^{(i)}_{k}\notin \mbZ$ for all $i, j, k$} \right\}.
\]
We impose the condition that $\nu^{(i)}_{j}-\nu^{(i)}_{k}\notin \mbZ\setminus \{0\}$ in order to apply Proposition \ref{variation}, and further require that $\nu^{(i)}_{j}-\nu^{(i)}_{k}\neq 0$ so that parabolic structures of $\mc(E, \nabla)$ are uniquely determined.
\begin{theorem}\label{MT2}
The following diagrams are commutative;
\[
\begin{tikzcd}
	\mcM^0\arrow[r, "\mc_\beta"]\arrow[d]&\mcM^0\arrow[d]\\
	\mcN^0\arrow[r,"w"]&\mcN^0,
\end{tikzcd} \quad 
\begin{tikzcd}
	\mcM^0\arrow[r, "\mc_\beta"]\arrow[d, "\tilde{\psi}_{\boldsymbol{\nu}}"]&\mcM^0\arrow[d, "\tilde{\psi}_{w(\boldsymbol{\nu})}"]\\
	\mcS^0\arrow[r, "\phi_{w}"]&\mcS^0.
\end{tikzcd}
\]
where $\mcM^0=\mcM\times_{\mcN}\mcN^0$ and  $\mcS^0=\mcS\times_{\mcN}\mcN^0$. 
\end{theorem}
The commutativity of the left diagram follows by Proposition \ref{variation}. We can show the commutativity of the right diagram in a same manner of Section 6 in \cite{Ma2}. Here, we outline the proof. For $(E, \nabla, l_*)\in M(\boldsymbol{\nu})$, we set
\begin{align*}
	(F, \nabla_F):=&\pi_1^*(E, \nabla_E)\otimes (\mcO, d+\beta),\\
	(E', \nabla_{E'}):=&\mc_{\beta}(E ,\nabla)=(\mbR^1\pi_{2*}\rMDR, \nabla^{\GM}_{\tmid}).
\end{align*}
We can prove Theorem \ref{MT2} in the following three steps.
\begin{itemize}\setlength{\itemsep}{0cm}
	\item[Step 1.]  Find a local basis of $E'$. 
	\item[Step 2.] Express $\nabla'$ for the above basis. 
	\item[Step 3.] Write down the apparent singularity $\overline{q}$ and the dual parameter $\overline{p}$ by $q$ and $p$.
\end{itemize}
By the definition of $\MDR_{X/Z}$, there is an exact sequence
\[
0 \rightarrow \pi_{2*}F \rightarrow \pi_{2*}(F'\otimes \xjzq{1}) \rightarrow E' \rightarrow \mbR^1\pi_{2*}F.
\]
We can check that $\mbR^1\pi_{2*}F=0$. Hence, $E'$ is isomorphic to the quotient of $\pi_{2*}(F'\otimes \xjzq{1})$ by the image of $\pi_{2*}F$.  In Appendix \ref{mcdata}, we summarize a basis of $\pi_{2*}(F'\otimes \xjzq{1})$ and $E'$, the image of $\pi_{2*}F$, and the matrix form for the basis of $E'$. 

\appendix\section{ Normal form of logarithmic connection}\label{abform}
In this appendix, we give entries of normal forms in Section \ref{normalform} explicitly.  Let $\nu^{(i)}=(\nu^{(i)}_1, \ldots, \nu^{(i)}_{s_i})$  and  
\[
f^{(n)}_m(x_1, \ldots, x_n)=\sum_{1\leq i_1<\cdots<i_m\leq n}x_{i_1}\cdots x_{i_m}.
\]

\begin{itemize}
	
\item$\bm{r}=(2^2, 1^4, 1^4)$
\item[$\circ$] $A(z)$	
\begin{align*}
	a_{10}=&-\nu^{(0)}_{1}\nu^{(0)}_{2},\\
	a_{11}=&-p^2-\tfrac{1}{2}p\nu_0-\tfrac{1}{4}\nu_0^2-\tfrac{1}{2}\nu_0\nu_1-f^{(4)}_2(\nu^{(1)}),\\
	a_{21}=&\tfrac{1}{8}\nu_0^3+\tfrac{1}{4}\nu^2_0\nu_1+\tfrac{1}{2}\nu_0f^{(4)}_2(\nu^{(1)})+f^{(4)}_3(\nu^{(1)})+(\nu_0+\nu_1)(p^2+\tfrac{1}{2}\nu_0p),\\
	(q-1)\mu_1=&(p-\nu^{(1)}_{1})(p-\nu^{(1)}_{2})(p-\nu^{(1)}_{3})(p-\nu^{(1)}_{4}),\\
	a_{1\infty}=&\mu_0-\nu_2-f^{(4)}_2(\nu^{(2)})-1,\\
	a_{2\infty}=&\nu_2\mu_0-f^{(4)}_3(\nu^{(2)}),\\
	a_{3\infty}=&-\mu_0^2+f^{(4)}_2(\nu^{(2)})\mu_0-f^{(4)}_4(\nu^{(2)}),\\
	q\mu_0=&(p+\nu^{(0)}_{1})(p+\nu^{(0)}_{2}).
\end{align*}
\item[$\circ$] $B(w)$	
\begin{align*}
	b_{10}=&-\tilde{p}^2-f^{(4)}_2(\nu^{(2)}),\\
	b_{20}=&-\nu_2\tilde{p}^2-f^{(4)}_3(\nu^{(2)}),\\
	\tilde{q}\lambda_2=&(\tilde{p}+\nu^{(2)}_{1})(\tilde{p}+\nu^{(2)}_{2})(\tilde{p}+\nu^{(2)}_{3})(\tilde{p}+\nu^{(2)}_{4}),\\
	b_{11}=&-\tilde{p}^2-\tfrac{1}{2}\nu_0\tilde{p}-\tfrac{1}{4}\nu_0^2-\tfrac{1}{2}\nu_0\nu_1-f^{(4)}_2(\nu^{(1)})-\nu^{(0)}_{1}\nu^{(0)}_{2}(\tilde{q}-1),\\
	b_{21}=&\tfrac{1}{8}\nu_0^3+\tfrac{1}{4}\nu_0^2\nu_1+\tfrac{1}{2}\nu_0f^{(4)}_2(\nu^{(1)})+f^{(4)}_3(\nu^{(1)})-(\nu_2+1)(\tilde{p}^2+\tfrac{1}{2}\nu_0\tilde{p}+\nu^{(0)}_1\nu^{(0)}_2(\tilde{q}-1))),\\
	(\tilde{q}-1)\lambda_{1}=&(\nu^{(0)}_{1}\nu^{(0)}_{2})^2(\tilde{q}-1)^2+(\tilde{p}-\nu^{(1)}_{1})(\tilde{p}-\nu^{(1)}_{2})(\tilde{p}-\nu^{(1)}_{3})(\tilde{p}-\nu^{(1)}_{4})\\
	&+(2\tilde{p}^2-\nu_1\tilde{p}+\tfrac{1}{4}\nu_0^2+\tfrac{1}{2}\nu_0\nu_1+f^{(4)}_2(\nu^{(1)}))\nu^{(0)}_{1}\nu^{(0)}_{2}(\tilde{q}-1),\\
	b_{1\infty}=&-(1+\nu^{(0)}_{1})(1+\nu^{(0)}_{2}).
\end{align*}
	
\item $\bm{r}=(3^2, 2^3, 1^6)$
\item[$\circ$] $A(z)$	
\begin{align*}
	a_1(z)=&a_{1\infty}z(z-1)+a_{11}z-a_{10}(z-1),\\
	a_2(z)=&z(a_{2\infty}(z-1)+a_{21}),\\
	a_{10}=&-\nu_{00}\nu_{01},\\
	a_{11}=&-\tfrac{4}{9}\nu_0^2-\nu_0\nu_1-\tfrac{1}{2}\nu_1^2-f_2(\nu^{(1)}),\\
	a_{21}=&\tfrac{1}{108}(4\nu_0+3\nu_1)(8\nu_0^2+18\nu_0\nu_1+9\nu_1^2+18f_2(\nu^{(1)}))+f_3(\nu^{(1)}),\\
	a_{1\infty}=&-\nu_2-f_2(\nu^{(2)})+a_{6}-a_{7}+a_9-1,\\
	a_{2\infty}=&(a_{6}-a_{7}+a_9)\nu_2+f_3(\nu^{(2)})-a_{8},\\
	a_{3}=&a_{6}a_9-(1+\nu_2+a_{1\infty})(a_{6}-a_{7}+a_9)-\nu_2a_8-f_4(\nu^{(2)}),\\
	a_{4}=&(1+\nu_2+a_{1\infty})a_8+a_{2\infty}(a_{7}-a_9)+\nu_2a_6a_9-f_5(\nu^{(2)}),\\
	a_{5}=&a_{2\infty}a_8-(1+\nu_2+a_{1\infty})a_6a_9-f_6(\nu^{(2)}),\\
	a_{6}=&\tfrac{2}{9}\nu_0^2+\tfrac{1}{2}\nu_0\nu_1+\tfrac{1}{4}\nu_1^2+\nu^{(0)}_1\nu^{(0)}_2,\\
	a_{7}=&-p^2-\tfrac{1}{3}\nu_0p-\tfrac{1}{6}\nu_0\nu_1-\tfrac{1}{9}\nu_0^2-f^{(3)}_2(\nu^{(1)}),\\
	(q-1)a_{8}=&(p-\nu^{(1)}_{1})(p-\nu^{(1)}_{1})(p-\nu^{(1)}_{2}),\\
	qa_9=&(p+\nu^{(0)}_{1})(p+\nu^{(0)}_{2}).
\end{align*}
\item[$\circ$] $B(w)$	
\begin{align*}
	b_1(w)=&b_{1\infty}w(w-1)+b_{11}w-b_{10}(w-1)\\
	b_i(w)=&b_{i1}w-b_{i0}(w-1)\quad (i>1)\\
	b_{10}=&-\tilde{p}^2-f^{(6)}_2(\nu^{(2)}),\\
	b_{20}=&-\nu_2\tilde{p}^2-f^{(6)}_3(\nu^{(2)}),\\
	b_{30}=&\tilde{p}^4+f^{(6)}_2(\nu^{(2)})\tilde{p}^2-f^{(6)}_4(\nu^{(2)}),\\
	b_{40}=&\nu_2\tilde{p}^4+f^{(6)}_3(\nu^{(2)})\tilde{p}^2-f^{(6)}_5(\nu^{(2)}),\\
	\tilde{q}\lambda_2=&(\tilde{p}+\nu^{(2)}_{1})(\tilde{p}+\nu^{(2)}_{2})(\tilde{p}+\nu^{(2)}_{3})(\tilde{p}+\nu^{(2)}_{4})(\tilde{p}+\nu^{(2)}_{5})(\tilde{p}+\nu^{(2)}_{6}),\quad (\lambda_2=b_{50})\\
	b_{11}=&-\tfrac{4}{9}\nu_0^2-\nu_0\nu_1-\tfrac{1}{2}\nu_1^2-f^{(3)}_2(\nu^{(1)}),\\
	b_{21}=&\tfrac{1}{6}(4+\nu_1+4\nu_2)b_{11}+f^{(3)}_3(\nu^{(1)}),\\
	b_{31}=&-(b_{11}-b_9+(\tfrac{1}{3}\nu_0+\tfrac{1}{2}\nu_1)^2)b_6-(\tilde{p}+\tfrac{2}{3}\nu_0+\tfrac{1}{2}\nu_1)b_7-(\tilde{q}-1)b_8\\
	b_{41}=&-\left((\nu_0+\nu_1)b_9+(\tfrac{1}{3}\nu_0+\tfrac{1}{2}\nu_1-\nu^{(1)}_{1})(\tfrac{1}{3}\nu_0+\tfrac{1}{2}\nu_1-\nu^{(1)}_{2})(\tfrac{1}{3}\nu_0+\tfrac{1}{2}\nu_1-\nu^{(1)}_{3})\right)b_6\\
	&+(\nu_0+\nu_1)(\tilde{p}+\tfrac{2}{3}\nu_0+\tfrac{1}{2}\nu_1)b_7+(\nu_0+\nu_1)(\tilde{q}-1)b_8\\
	b_{5}=&((\tilde{p}-\tfrac{1}{3}\nu_0-\tfrac{1}{2}\nu_1)\lambda_1-\nu^{(0)}_1\nu^{(0)}_2b_9)b_6\\
	&+(\lambda_1+\nu^{(0)}_1\nu^{(0)}_2(\nu_0+\nu_1))b_7-(b_9+(\nu_0+\nu_1)(\tilde{p}-\tfrac{1}{3}\nu_0-\tfrac{1}{2}\nu_1))b_8\\
	(\tilde{q}-1)\lambda_{1}=&\nu^{(0)}_{1}\nu^{(0)}_{2}(\tilde{q}-1)(\tilde{p}-\tfrac{1}{3}\nu_0-\nu_1)+(\tilde{p}-\nu^{(1)}_{1})(\tilde{p}-\nu^{(1)}_{2})(\tilde{p}-\nu^{(1)}_{3}),\\
	b_{6}=&\tfrac{2}{9}\nu_0^2+\tfrac{1}{2}\nu_0\nu_1+\tfrac{1}{4}\nu_1^2+\nu^{(0)}_1\nu^{(0)}_2\\
	b_7=&-\lambda_{1}+\tfrac{1}{3}\nu_0b_9\\
	b_8=&\nu^{(0)}_1\nu^{(0)}_2b_9\\
	b_{9}=&-\tilde{p}^2-\tfrac{1}{3}\nu_0\tilde{p}-\tfrac{1}{6}\nu_0\nu_1-\tfrac{1}{9}\nu_0^2-f^{(3)}_2(\nu^{(1)})-\nu^{(0)}_1\nu^{(0)}_2(\tilde{q}-1)
\end{align*}
\end{itemize}

\section{The Picard group of the spaces of initial conditions for  the difference Painlev\'e equations of type $R^\perp$}
\subsection{The Picard group and the root lattice}\label{picbasis}
\begin{itemize}
\item $R=A^{(1)*}_1, R^\perp=E^{(1)}_7$: 

\item[$\circ$] $\Pic X_{\boldsymbol{\nu}}\cong \mbZ H_1\oplus \mbZ H_2 \oplus \bigoplus_{i=1}^8 \mbZ \mcF_i$
\[
H_i^2=0, \quad \mcF_j^2=-1, \quad  H_1.H_2=1, \quad H_1. \mcF_j=H_2. \mcF_j=\mcF_i. \mcF_j=0\ (i \neq j)
\]
\item[$\circ$] $Q(E^{(1)}_7)=\bigoplus_{i=0}^7 \mbZ \alpha_i$
\[
\begin{array}{llll}
\alpha_0=H_1-H_2, & \alpha_1=\mcF_3-\mcF_4, & \alpha_2=\mcF_2-\mcF_3, & \alpha_3=\mcF_1-\mcF_2, \\
 \alpha_4=H_2-\mcF_1-\mcF_5,& \alpha_5=\mcF_5-\mcF_6, & \alpha_6=\mcF_6-\mcF_7, & \alpha_7=\mcF_7-\mcF_8.
\end{array}
\]

\item[$\circ$]$\Pic S_{\boldsymbol{\nu}}\cong \mbZ H \oplus \bigoplus_{i=0}^{2}\bigoplus_{j=1}^{s_i} \mbZ \mcE^{(i)}_j$
\[
H_i=\left\{
\begin{array}{ll}
	H-\mcE^{(0)}_2&i=1\\H-\mcE^{(0)}_1&i=2,
\end{array}\right. \quad 
\mcF_j=\left\{
\begin{array}{ll}
	\mcE^{(1)}_{j}&j=1, \cdots, 4\\\mcE^{(2)}_{j-4}&j=5, \cdots, 8
\end{array}\right.
\]
\end{itemize}
\begin{itemize}
\item $R=A^{(1)**}_0, R^\perp=E^{(1)}_8$: 
\item[$\circ$]$\Pic X\cong \bigoplus_{i=0}^9 \mbZ \mcG_i$
\[
\mcG_0^2=1, \quad \mcG_i^2=-1\ (i=1, \ldots, 9), \quad \mcG_i. \mcG_j=0\ (i \neq j)
\]
\item[$\circ$] $Q(E^{(1)}_8)=\bigoplus_{i=0}^8 \mbZ \alpha_i$
\[
\alpha_0=\mcG_8-\mcG_9, \quad \alpha_i=\mcG_i-\mcG_{i+1} \ (i=1, \ldots ,7), \quad \alpha_8=\mcG_0-\mcG_1-\mcG_2-\mcG_3
\]
\item[$\circ$]$\Pic S_{\boldsymbol{\nu}}\cong \mbZ H \oplus \bigoplus_{i=0}^{2}\bigoplus_{j=1}^{s_i} \mbZ \mcE^{(i)}_j$
\[
\mcG_i=\left\{
\begin{array}{ll}
3H-2\mcE^{(0)}_1-\mcE^{(0)}_2-\mcE^{(1)}_1-\mcE^{(1)}_2-\mcE^{(1)}_3&i=0, \\ 
H-\mcE^{(0)}_1-\mcE^{(1)}_{4-i}&i=1, 2, 3, \\
\mcE^{(2)}_{i-3} &i=4, \ldots, 9.
\end{array}\right.
\]
\end{itemize}

\subsection{Periods of the basis of $Q(R^\perp)$}

\begin{itemize}
\item $R=A^{(1)*}_1, R^\perp=E^{(1)}_7$: 
\[
\begin{array}{lll}
	\chi(\alpha_0)=\nu^{(0)}_2-\nu^{(0)}_1
	&\chi(\alpha_1)=\nu^{(1)}_4-\nu^{(1)}_3, 
	& \chi(\alpha_2)=\nu^{(1)}_3-\nu^{(1)}_2, \\
	\chi(\alpha_3)=\nu^{(1)}_2-\nu^{(1)}_1, 
	&\chi(\alpha_4)=\nu^{(0)}_1+\nu^{(1)}_1+\nu^{(2)}_1,
	& \chi(\alpha_5)=\nu^{(2)}_2-\nu^{(2)}_1, \\
	 \chi(\alpha_6)=\nu^{(2)}_3-\nu^{(2)}_2, 
	&\chi(\alpha_7)=\nu^{(2)}_4-\nu^{(2)}_3,
\end{array}
\]
\end{itemize}
\begin{itemize}
\item $R=A^{(1)**}_0, R^\perp=E^{(1)}_8$: 
\[
\begin{array}{lll}
	\chi(\alpha_0)=\nu^{(2)}_6-\nu^{(2)}_5
	&\chi(\alpha_1)=\nu^{(1)}_3-\nu^{(1)}_2, 
	& \chi(\alpha_2)=\nu^{(1)}_2-\nu^{(1)}_1, \\
	\chi(\alpha_3)=\nu^{(0)}_1+\nu^{(1)}_1+\nu^{(2)}_1,
	&\chi(\alpha_4)=\nu^{(2)}_2-\nu^{(2)}_1, 
	& \chi(\alpha_5)=\nu^{(2)}_3-\nu^{(2)}_2, \\
	\chi(\alpha_6)=\nu^{(2)}_4-\nu^{(2)}_3, 
	&\chi(\alpha_7)=\nu^{(2)}_5-\nu^{(2)}_4,
	&\chi(\alpha_8)=\nu^{(1)}_2-\nu^{(1)}_1,
\end{array}
\]
\end{itemize}

\section{The action of $\widetilde{W}(R^\perp)$}

\subsection{The action on $\Pic(X_{\boldsymbol{\nu}})$}\label{actpicX}

\begin{itemize}
\item $R=A^{(1)*}_1, R^\perp=E^{(1)}_7$: $\Pic X\cong \mbZ H_1\oplus \mbZ H_2 \oplus \bigoplus_{i=1}^8 \mbZ \mcF_i$
\[
\begin{array}{ll}
w_0(\mcF)=
\left\{
\begin{array}{ll}
H_{(12)(i)} & \mcF=H_i\\\mcF_j &\mcF=\mcF_j,
\end{array}
\right. &
w_1(\mcF)=
\left\{
\begin{array}{ll}
H_{i} & \mcF=H_i\\\mcF_{(34)(j)} &\mcF=\mcF_j,
\end{array}
\right.\\
w_2(\mcF)=
\left\{
\begin{array}{ll}
H_{i} & \mcF=H_i\\\mcF_{(23)(j)} &\mcF=\mcF_j,
\end{array}
\right. &
w_3(\mcF)=
\left\{
\begin{array}{ll}
H_{i} & \mcF=H_i\\\mcF_{(12)(j)} &\mcF=\mcF_j,
\end{array}
\right.\\
w_5(\mcF)=
\left\{
\begin{array}{ll}
H_i & \mcF=H_i\\\mcF_{(56)(j)} &\mcF=\mcF_j,
\end{array}
\right. &
w_6(\mcF)=
\left\{
\begin{array}{ll}
H_{i} & \mcF=H_i\\\mcF_{(67)(j)} &\mcF=\mcF_j,
\end{array}
\right.\\
w_7(\mcF)=
\left\{
\begin{array}{ll}
H_{i} & \mcF=H_i\\\mcF_{(78)(j)} &\mcF=\mcF_j,
\end{array}
\right. &
\sigma(\mcF)=
\left\{
\begin{array}{ll}
H_{i} & \mcF=H_i\\\mcF_{(15)(26)(37)(48)(j)} &\mcF=\mcF_j,
\end{array}
\right.
\end{array}
\]
\[
w_4(\mcF)=
\left\{
\begin{array}{ll}
H_1+H_2-\mcF_1-\mcF_5 & \mcF=H_1\\
H_2-\mcF_5 &\mcF=\mcF_1\\
H_2-\mcF_1 &\mcF=\mcF_5\\
\mcF & \mcF\neq H_1, \mcF_1, \mcF_5,
\end{array}
\right.
\]
\end{itemize}
\begin{itemize}
\item $R=A^{(1)**}_0, R^\perp=E^{(1)}_8$: $\Pic X\cong \bigoplus_{i=0}^9 \mbZ \mcG_i$
\[
w_i(\mcG_j)=\mcG_{(i\ i+1)(j)}, \quad (i=1, \ldots, 7), \quad w_0(\mcG_j)=\mcG_{(89)(j)}
\]
\[
w_8(\mcG_j)=
\left\{
\begin{array}{ll}
2\mcG_0-\mcG_1-\mcG_2-\mcG_3 &j=0\\
\mcG_0-\mcG_2-\mcG_3 &j=1\\
\mcG_0-\mcG_1-\mcG_3 &j=2\\
\mcG_0-\mcG_1-\mcG_2 &j=3\\
\mcG_j &j\neq 0, 1,2,3
\end{array}
\right.
\]
\end{itemize}

\subsection{The action on $\Pic (S_{\boldsymbol{\nu}})$}\label{actpicS}

\begin{itemize}
	\item $R=A^{(1)*}_1, R^\perp=E^{(1)}_7$: $\Pic S\cong \mbZ H \oplus \bigoplus_{i=0}^{2}\bigoplus_{j=1}^{s_i} \mbZ \mcE^{(i)}_j$
	\[
	w_i(H)=\left\{
	\begin{array}{ll}
		H&i\neq 4, \\2H-\mcE^{(0)}_1-\mcE^{(1)}_1-\mcE^{(2)}_1 &i=4
	\end{array}\right.
	\]
	\[
	\begin{array}{ll}
		w_0(\mcE^{(i)}_j)=
		\left\{
		\begin{array}{ll}
			\mcE^{(i)}_{(12)(j)}&i=0, \\\mcE^{(i)}_j&i\neq 0
		\end{array}\right.&
		w_1(\mcE^{(i)}_j)=
		\left\{
		\begin{array}{ll}
			\mcE^{(i)}_{(34)(j)}&i=1, \\\mcE^{(i)}_j&i\neq 1
		\end{array}\right.\\
		w_2(\mcE^{(i)}_j)=
		\left\{
		\begin{array}{ll}
			\mcE^{(i)}_{(23)(j)}&i=1, \\\mcE^{(i)}_j&i\neq 1
		\end{array}\right. &
		w_3(\mcE^{(i)}_j)=
		\left\{
		\begin{array}{ll}
			\mcE^{(i)}_{(12)(j)}&i=1, \\\mcE^{(i)}_j&i\neq 1
		\end{array}\right. \\
		w_5(\mcE^{(i)}_j)=
		\left\{
		\begin{array}{ll}
			\mcE^{(i)}_{(12)(j)}&i=2, \\\mcE^{(i)}_j&i\neq 2
		\end{array}\right.&
		w_6(\mcE^{(i)}_j)=
		\left\{
		\begin{array}{ll}
			\mcE^{(i)}_{(23)(j)}&i=2, \\\mcE^{(i)}_j&i\neq 2
		\end{array}\right.\\
		w_7(\mcE^{(i)}_j)=
		\left\{
		\begin{array}{ll}
			\mcE^{(i)}_{(34)(j)}&i=2, \\\mcE^{(i)}_j&i\neq 2
		\end{array}\right.&
		\sigma(\mcE^{(i)}_j)=\mcE^{(12)(i)}_j
	\end{array}
	\]
	\[
	w_4(\mcE^{(i)}_j)=
	\left\{
	\begin{array}{ll}
		H-\mcE^{(1)}_1-\mcE^{(2)}_1 &i=0, j=1\\
		H-\mcE^{(0)}_1-\mcE^{(2)}_1&i=1, j=1\\
		H-\mcE^{(0)}_1-\mcE^{(1)}_1&i=2, j=1\\
		\mcE^{(i)}_j & j\neq 1
	\end{array}
	\right.
	\]
\end{itemize}
\begin{itemize}
	\item $R=A^{(1)**}_0, R^\perp=E^{(1)}_8$: $\Pic S\cong \mbZ H \oplus \bigoplus_{i=0}^{2}\bigoplus_{j=1}^{s_i} \mbZ \mcE^{(i)}_j$
	\[
	w_i(H)=\left\{
	\begin{array}{ll}
		H&i\neq 3, \\2H-\mcE^{(0)}_1-\mcE^{(1)}_1-\mcE^{(2)}_1 &i=3
	\end{array}\right.
	\]
	\[
	\begin{array}{ll}
		w_0(\mcE^{(i)}_j)=
		\left\{
		\begin{array}{ll}
			\mcE^{(i)}_{(56)(j)}&i=2, \\\mcE^{(i)}_j&i\neq 2
		\end{array}\right.&
		w_1(\mcE^{(i)}_j)=
		\left\{
		\begin{array}{ll}
			\mcE^{(i)}_{(23)(j)}&i=1, \\\mcE^{(i)}_j&i\neq 1
		\end{array}\right.\\
		w_2(\mcE^{(i)}_j)=
		\left\{
		\begin{array}{ll}
			\mcE^{(i)}_{(12)(j)}&i=1, \\\mcE^{(i)}_j&i\neq 1
		\end{array}\right. &
		w_4(\mcE^{(i)}_j)=
		\left\{
		\begin{array}{ll}
			\mcE^{(i)}_{(12)(j)}&i=2, \\\mcE^{(i)}_j&i\neq 2
		\end{array}\right. \\
		w_5(\mcE^{(i)}_j)=
		\left\{
		\begin{array}{ll}
			\mcE^{(i)}_{(23)(j)}&i=2, \\\mcE^{(i)}_j&i\neq 2
		\end{array}\right.&
		w_6(\mcE^{(i)}_j)=
		\left\{
		\begin{array}{ll}
			\mcE^{(i)}_{(34)(j)}&i=2, \\\mcE^{(i)}_j&i\neq 2
		\end{array}\right.\\
		w_7(\mcE^{(i)}_j)=
		\left\{
		\begin{array}{ll}
			\mcE^{(i)}_{(45)(j)}&i=2, \\\mcE^{(i)}_j&i\neq 2
		\end{array}\right.&
		w_8(\mcE^{(i)}_j)=
		\left\{
		\begin{array}{ll}
			\mcE^{(i)}_{(12)(j)}&i=0, \\\mcE^{(i)}_j&i\neq 0
		\end{array}\right.
	\end{array}
	\]
	\[
	w_3(\mcE^{(i)}_j)=
	\left\{
	\begin{array}{ll}
		H-\mcE^{(1)}_1-\mcE^{(2)}_1 &i=0, j=1\\
		H-\mcE^{(0)}_1-\mcE^{(2)}_1&i=1, j=1\\
		H-\mcE^{(0)}_1-\mcE^{(1)}_1&i=2, j=1\\
		\mcE^{(i)}_j & j\neq 1
	\end{array}
	\right.
	\]
\end{itemize}

\subsection{The action on $\mcN$}\label{actionN0}
We denote the reflection corresponding to the node $i$ by $w_i$.
\begin{itemize}
\item $R=A^{(1)*}_1, R^\perp=E^{(1)}_7$:  $\gamma=\nu^{(0)}_{1}+\nu^{(1)}_{1}+\nu^{(2)}_{1}=\chi(\alpha_4)$  
\[
\begin{array}{ll}
	w_0(\nu^{(i)}_j)=
	\left\{
	\begin{array}{ll}
		\nu^{(i)}_{(12)(j)}&i=0, \\\nu^{(i)}_j&i\neq 0
	\end{array}\right.&
	w_1(\nu^{(i)}_j)=
	\left\{
	\begin{array}{ll}
		\nu^{(i)}_{(34)(j)}&i=1, \\\nu^{(i)}_j&i\neq 1
	\end{array}\right.\\
	w_2(\nu^{(i)}_j)=
	\left\{
	\begin{array}{ll}
		\nu^{(i)}_{(23)(j)}&i=1, \\\nu^{(i)}_j&i\neq 1
	\end{array}\right. &
	w_3(\nu^{(i)}_j)=
	\left\{
	\begin{array}{ll}
		\nu^{(i)}_{(12)(j)}&i=1, \\\nu^{(i)}_j&i\neq 1
	\end{array}\right. \\
	w_4(\nu^{(i)}_j)=
	\left\{
	\begin{array}{cc}
		\nu^{(i)}_{j}-\frac{s_i-1}{s_i}\gamma & j=1\\
		\nu^{(i)}_{j}+\frac{1}{s_i}\gamma &  j\neq 1\\
	\end{array}
	\right.&
	w_5(\nu^{(i)}_j)=
	\left\{
	\begin{array}{ll}
		\nu^{(i)}_{(12)(j)}&i=2, \\\nu^{(i)}_j&i\neq 2
	\end{array}\right.\\
	w_6(\nu^{(i)}_j)=
	\left\{
	\begin{array}{ll}
		\nu^{(i)}_{(23)(j)}&i=2, \\\nu^{(i)}_j&i\neq 2
	\end{array}\right.&
	w_7(\nu^{(i)}_j)=
	\left\{
	\begin{array}{ll}
		\nu^{(i)}_{(34)(j)}&i=2, \\\nu^{(i)}_j&i\neq 2
	\end{array}\right.\\
	\sigma(\nu^{(i)}_{j})=
	\left\{
	\begin{array}{cc}
		\nu^{(0)}_{j} & i=0\\
		\nu^{(2)}_{j}+\frac{1}{4}(\nu_1-\nu_2) & i=1\\
		\nu^{(1)}_{j}+\frac{1}{4}(\nu_2-\nu_1) & i=2,
	\end{array}
	\right. 
\end{array}
\]
\end{itemize}
\begin{itemize}
\item $R=A^{(1)**}_0, R^\perp=E^{(1)}_8$: 
$\gamma=\nu^{(0)}_{1}+\nu^{(1)}_{1}+\nu^{(2)}_{1}=\chi(\alpha_3)$  
\[
\begin{array}{ll}
	w_0(\nu^{(i)}_j)=
	\left\{
	\begin{array}{ll}
		\nu^{(i)}_{(56)(j)}&i=2, \\\nu^{(i)}_j&i\neq 2
	\end{array}\right.&
	w_1(\nu^{(i)}_j)=
	\left\{
	\begin{array}{ll}
		\nu^{(i)}_{(23)(j)}&i=1, \\\nu^{(i)}_j&i\neq 1
	\end{array}\right.\\
	w_2(\nu^{(i)}_j)=
	\left\{
	\begin{array}{ll}
		\nu^{(i)}_{(12)(j)}&i=1, \\\nu^{(i)}_j&i\neq 1
	\end{array}\right. &
	w_3(\nu^{(i)}_j)=
	\left\{
	\begin{array}{cc}
		\nu^{(i)}_{j}-\frac{s_i-1}{s_i}\gamma & j=1\\
		\nu^{(i)}_{j}+\frac{1}{s_i}\gamma &  j\neq 1\\
	\end{array}
	\right.\\
	w_4(\nu^{(i)}_j)=
	\left\{
	\begin{array}{ll}
		\nu^{(i)}_{(12)(j)}&i=2, \\\nu^{(i)}_j&i\neq 2
	\end{array}\right. &
	w_5(\nu^{(i)}_j)=
	\left\{
	\begin{array}{ll}
		\nu^{(i)}_{(23)(j)}&i=2, \\\nu^{(i)}_j&i\neq 2
	\end{array}\right.\\
	w_6(\nu^{(i)}_j)=
	\left\{
	\begin{array}{ll}
		\nu^{(i)}_{(34)(j)}&i=2, \\\nu^{(i)}_j&i\neq 2
	\end{array}\right.&
	w_7(\nu^{(i)}_j)=
	\left\{
	\begin{array}{ll}
		\mcE^{(i)}_{(45)(j)}&i=2, \\\nu^{(i)}_j&i\neq 2
	\end{array}\right.\\
	w_8(\nu^{(i)}_j)=
	\left\{
	\begin{array}{ll}
		\nu^{(i)}_{(12)(j)}&i=0, \\\nu^{(i)}_j&i\neq 0
	\end{array}\right.
\end{array}
\]
\end{itemize}
\subsection{The action on $\mcS$}
See Section \ref{actS}. 
\section{Realization of affine Weyl group symmetry}

\subsection{Date for middle convolution}\label{apprealize}
\begin{itemize}
\item $R=A^{(1)*}_1, R^\perp=E^{(1)}_7$: 
\item[$\circ$] $f_{w_0}(E, \nabla, l_*)=(E, \nabla, l'_*)$
\[
l'^{(i)}_{j}:=
\left\{
\begin{array}{cc}
\Ker(\res_0\nabla-\nu^{(0)}_{1}\id\colon E|_{0}\longrightarrow E|_{0})& (i, j)=(0, 1)\\
l^{(i)}_{j}& (i, j)\neq (0, 1)
\end{array}
\right.
\]
\item[$\circ$] $f_{w_k}(E, \nabla, l_*)=(E, \nabla, l'_*)$ for $k=1, 2, 3$
\[
l'^{(i)}_{j}:=
\left\{
\begin{array}{cc}
	\Ker( E|_{1}\rightarrow E|_{1}/l^{(1)}_{5-k}\overset{\res_1\nabla-\nu^{(1)}_{4-k}\id}{\longrightarrow} E|_{1}/l^{(1)}_{5-k})& (i, j)=(1, 4-k)\\
	l^{(i)}_{j}& (i, j)\neq (1, 4-k)
\end{array}
\right.
\]
\item[$\circ$] $f_{w_4}(E, \nabla, l_*)=(\mc_{\beta}(E, \nabla), l'_*)$, see Section \ref{realizecn}
\item[$\circ$] $f_{w_k}(E, \nabla, l_*)=(E, \nabla, l'_*)$ for $k=5, 6, 7$
\[
l'^{(i)}_{j}:=
\left\{
\begin{array}{cc}
	\Ker( E|_{\infty}\rightarrow E|_{\infty}/l^{(2)}_{k-3}\overset{\res_\infty\nabla-\nu^{(2)}_{k-4}\id}{\longrightarrow} E|_{\infty}/l^{(2)}_{k-3})& (i, j)=(2, k-4)\\
	l^{(i)}_{j}& (i, j)\neq (2, k-4)
\end{array}
\right.
\]
\item[$\circ$] $f_{\sigma}(E, \nabla, l_*)=(s^*(E, \nabla), l'_*)$, see Section \ref{realizeaut}
\item $R=A^{(1)**}_0, R^\perp=E^{(1)}_8$: 
\item[$\circ$] $f_{w_0}(E, \nabla, l_*)=(E, \nabla, l'_*)$
\[
l'^{(i)}_{j}:=
\left\{
\begin{array}{cc}
	\Ker(\res_\infty\nabla-\nu^{(2)}_{5}\id\colon E|_{2}\longrightarrow E|_{2})& (i, j)=(0, 5)\\
	l^{(i)}_{j}& (i, j)\neq (0, 1)
\end{array}
\right.
\]
\item[$\circ$] $f_{w_k}(E, \nabla, l_*)=(E, \nabla, l'_*)$ for $k=1, 2$
\[
l'^{(i)}_{j}:=
\left\{
\begin{array}{cc}
	\Ker( E|_{1}\rightarrow E|_{1}/l^{(1)}_{4-k}\overset{\res_1\nabla-\nu^{(1)}_{3-k}\id}{\longrightarrow} E|_{1}/l^{(1)}_{4-k})& (i, j)=(1, 3-k)\\
	l^{(i)}_{j}& (i, j)\neq (1, 3-k)
\end{array}
\right.
\]
\item[$\circ$] $f_{w_3}(E, \nabla, l_*)=(\mc_{\beta}(E, \nabla), l'_*)$, see Section \ref{realizecn}
\item[$\circ$] $f_{w_k}(E, \nabla, l_*)=(E, \nabla, l'_*)$ for $k=4, 5, 6, 7$
\[
l'^{(i)}_{j}:=
\left\{
\begin{array}{cc}
	\Ker( E|_{\infty}\rightarrow E|_{\infty}/l^{(2)}_{k-2}\overset{\res_\infty\nabla-\nu^{(2)}_{k-3}\id}{\longrightarrow} E|_{\infty}/l^{(2)}_{k-2})& (i, j)=(2, k-3)\\
	l^{(i)}_{j}& (i, j)\neq (2, k-3)
\end{array}
\right.
\]
\item[$\circ$] $f_{w_8}(E, \nabla, l_*)=(E, \nabla, l'_*)$
\[
l'^{(i)}_{j}:=
\left\{
\begin{array}{cc}
	\Ker(\res_0\nabla-\nu^{(0)}_{1}\id\colon E|_{0}\longrightarrow E|_{0})& (i, j)=(0, 1)\\
	l^{(i)}_{j}& (i, j)\neq (0, 1)
\end{array}
\right.
\]
\end{itemize}
\subsection{Date for middle convolution}\label{mcdata}
For $\eta \in \pi_{2*}(F'\otimes \xjzq{1})$, Let $[\eta]$ be the image of $\eta$ by $\pi_{2*}(F'\otimes \xjzq{1})\rightarrow E'$.
\begin{itemize}
\item $R=A^{(1)*}_1, R^\perp=E^{(1)}_7$: 
\item[$\circ$] a basis of $\pi_{2*}(F'\otimes \xjzq{1})$ on $U_0$
\begin{align*}
\eta_1=&\nu^{(0)}_1e_1\otimes\tfrac{dy}{y}+e_2\otimes(\tfrac{dy}{y}-\tfrac{dy}{y-z}), \\
\eta_2=&((p+\nu^{(0)}_2)e_3+qe_4)\otimes(\tfrac{dy}{y}-\tfrac{dy}{y-z}), \\
\eta_3=&-\nu^{(1)}_1e_1\otimes\tfrac{dy}{y-1}+e_2\otimes(\tfrac{dy}{y-1}-\tfrac{dy}{y-z}), \\
\eta_4=&a_{11}e_1\otimes\tfrac{dy}{y-1}+((\tfrac{1}{2}\nu_0+\nu_1-\nu^{(1)}_1)e_2+e_3)\otimes(\tfrac{dy}{y-1}-\tfrac{dy}{y-z}), \\
\eta_5=&(p-\nu^{(1)}_2)(p-\nu^{(1)}_3)(p-\nu^{(1)}_4)e_1\otimes\tfrac{dy}{y-1}+(q-1)e_4\otimes(\tfrac{dy}{y-1}-\tfrac{dy}{y-z}), \\
\eta_6=&(\nu^{(2)}_1+1)e_1\otimes dy+e_2\otimes \tfrac{dy}{y-z},\\
\eta_7=&a_{1\infty}e_1\otimes dy+((\nu^{(2)}_1-\nu_2-1)e_2+e_3)\otimes\tfrac{dy}{y-z}, \\
\eta_8=&a_{2\infty}e_1\otimes dy+ (\nu^{(2)}_1e_3+e_4)\otimes \tfrac{dy}{y-z},\\
\eta_9=&e_1\otimes \tfrac{dy}{y-z}
\end{align*}
\item[$\circ$] a basis of $\pi_{2*}(F)$ on $U_0$: $e_1, e_2, e_3, ye_1$
\item[$\circ$] the image of the basis of $\pi_{2*}(F)$ 
\begin{align*}
\nabla_F(e_1)=&-\eta_1+\eta_3+\gamma \eta_9\\
\nabla_F(e_2)=&\nu^{(0)}_2\eta_1+\eta_4+\eta_7\\
\nabla_F(e_3)=&\eta_2+(\tfrac{1}{2}\nu_0+\nu_1-\nu^{(1)}_1)(p+\tfrac{1}{2}\nu_0+\nu^{(1)}_1)\eta_3-(p+\tfrac{1}{2}\nu_0+\nu^{(1)}_1)\eta_4-\eta_5+\eta_8\\
\nabla_F(e_4)=&-\tfrac{p+\nu^{(0)}_1}{q}\eta_2+\tfrac{p-\nu^{(1)}_1}{q-1}\eta_5+(\nu^{(2)}_1-\nu_2-1)(\mu_0+(\nu^{(2)}_1)^2)\eta_6\\
&-(\mu_0+(\nu^{(2)}_1)^2)\eta_7+\nu^{(2)}_1\eta_8\\
\nabla_F(ye_1)=&\eta_3+\eta_6+\gamma z \eta_9
\end{align*}
\item[$\circ$] a basis of $E'$ on $U_0$: $[\eta_6], \ldots, [\eta_{9}]$
\[
\nabla'([\eta_9], [\eta_6], [\eta_7], [\eta_8])=
([\eta_9], [\eta_6], [\eta_7], [\eta_8])
\begin{pmatrix}
*&*&\cdots &*\\1&*&\ddots&\vdots\\0&z&\ddots&\vdots\\0&0&*&*
\end{pmatrix}\frac{dz}{z(z-1)}
\]
\end{itemize}
\begin{itemize}
\item $R=A^{(1)**}_0, R^\perp=E^{(1)}_8$: 
\item[$\circ$] a basis of $\pi_{2*}(F'\otimes \xjzq{1})$
\begin{align*}
	\eta_1=&\nu^{(0)}_1e_1\otimes\tfrac{dy}{y}+e_2\otimes(\tfrac{dy}{y}-\tfrac{dy}{y-z}), \\
	\eta_2=&((\tfrac{2}{3}\nu_0+\tfrac{1}{2}\nu_1-\nu^{(0)}_1)e_3+e_4)\otimes(\tfrac{dy}{y}-\tfrac{dy}{y-z}), \\
	\eta_3=&((p+\nu^{(0)}_2)e_5+qe_6)\otimes(\tfrac{dy}{y}-\tfrac{dy}{y-z}), \\
	\eta_4=&-\nu^{(1)}_1e_1\otimes\tfrac{dy}{y-1}+e_2\otimes(\tfrac{dy}{y-1}-\tfrac{dy}{y-z}), \\
	\eta_5=&a_{11}e_1\otimes\tfrac{dy}{y-1}+((\tfrac{2}{3}\nu_0+\nu_1-\nu^{(1)}_1)e_2+e_3)\otimes(\tfrac{dy}{y-1}-\tfrac{dy}{y-z}), \\
	\eta_6=&((\tfrac{1}{3}\nu_0+\tfrac{1}{2}\nu_1-\nu^{(1)}_1)e_4+e_5)\otimes(\tfrac{dy}{y}-\tfrac{dy}{y-z}), \\
	\eta_7=&((p-\nu^{(1)}_2)(p-\nu^{(2)}_2)e_5+(q-1)e_6)\otimes(\tfrac{dy}{y-1}-\tfrac{dy}{y-z}), \\
	\eta_8=&(\nu^{(2)}_1+1)e_1\otimes dy+e_2\otimes \tfrac{dy}{y-z},\\
	\eta_9=&a_{1\infty}e_1\otimes dy+((\nu^{(2)}_1-\nu_2-1)e_2+e_3)\otimes\tfrac{dy}{y-z}, \\
	\eta_{10}=&a_{2\infty}e_1\otimes dy+ (\nu^{(2)}_1e_3+e_4)\otimes \tfrac{dy}{y-z},\\
	\eta_{11}=&a_{3}e_1\otimes dy+ (-a_6e_3+\nu^{(2)}_1e_4+e_5)\otimes \tfrac{dy}{y-z},\\
	\eta_{12}=&a_{4}e_1\otimes dy+ (a_7e_4+\nu^{(2)}_1e_5+e_6)\otimes \tfrac{dy}{y-z},\\
	\eta_{13}=&e_1\otimes \tfrac{dy}{y-z}
\end{align*}
\item[$\circ$] a basis of $\pi_{2*}(F)$ on $U_0$: $e_1, \ldots, e_6, ye_1$
\item[$\circ$] the image of the basis of $\pi_{2*}(F)$ 
\begin{align*}
	\nabla(e_1)=&-\eta_1+\eta_4+\gamma \eta_{13}\\
	\nabla(e_2)=&\nu^{(0)}_2\eta_1+\eta_5+\eta_9\\
	\nabla(e_3)=&\eta_2+(\tfrac{2}{3}\nu_0+\tfrac{1}{2}\nu_1+\nu^{(1)}_1)(\tfrac{2}{3}\nu_0+\nu_1-\nu^{(1)}_1)\eta_4-(\tfrac{2}{3}\nu_0+\tfrac{1}{2}\nu_1+\nu^{(1)}_1)\eta_5+\eta_{10}\\
	\nabla(e_4)=&-(\tfrac{2}{3}\nu_0+\tfrac{1}{2}\nu_1-\nu^{(0)}_2)\eta_2+\eta_6+\eta_{11}\\
	\nabla(e_5)=&\eta_3-(p+\tfrac{1}{3}\nu_0+\nu^{(1)}_1)\eta_6-\eta_7+\eta_{12}\\
	\nabla(e_6)=&-\tfrac{p+\nu^{(0)}_1}{q}\eta_3+\tfrac{p-\nu^{(1)}_1}{q-1}\eta_7\\
	&+(\nu^{(2)}_1-\nu_2-1)\left(\nu^{(2)}_1\mu_1+(\nu^{(2)}_1)^2(a_6-a_7+(\nu^{(2)}_1)^2)+\mu_0(a_6+(\nu^{(2)}_1)^2)\right)\eta_{8}\\
	&-\left(\nu^{(2)}_1\mu_1+(\nu^{(2)}_1)^2(a_6-a_7+(\nu^{(2)}_1)^2)+\mu_0(a_6+(\nu^{(2)}_1)^2)\right)\eta_{9}\\
	&+(\nu^{(2)}_1\mu_0+\mu_1+\nu^{(2)}_1(-a_7+(\nu^{(2)}_1)^2)\eta_{10}\\
	&-(\mu_0+(\nu^{(2)}_1)^2)\eta_{11}+\nu^{(2)}_1\eta_{12}\\
	\nabla(ye_1)=&\eta_4+\eta_8+\gamma z \eta_{13}
\end{align*}
\item[$\circ$] a basis of $E'$ on $U_0$: $[\eta_8], \ldots, [\eta_{13}]$
\[
\nabla'([\eta_{13}], [\eta_8], \ldots, [\eta_{12}])=
([\eta_{13}], [\eta_8], \ldots, [\eta_{12}])
\begin{pmatrix}
	*&&\cdots &&&*\\1&\ddots&&&&\\0&z&&&&\\0&0&z-1&&&\vdots\\0&0&0&z&\ddots&\\0&0&0&0&*&*\\
\end{pmatrix}\frac{dz}{z(z-1)}
\]
\end{itemize}


\begin{thebibliography}{9}
		\bibitem{Bo1}P. Boalch, 
		Irregular connections and Kac-Moody root systems,
		2008, arXiv:0806.1050. 
		\bibitem{Bo2}P. Boalch, Quivers and difference Painlev\'e equations, Groups and symmetries, 25--51, CRM Proc. Lecture Notes, 47, Amer. Math. Soc., Providence, RI, 2009.
		\bibitem{Bo3}P. Boalch, Simply-laced isomonodromy systems, Publ. Math. I.H.E.S. 116 (2012), no. 1, 1--68,
		\bibitem{Cr} W. Crawley-Boevey, On matrices in prescribed conjugacy classes with no common
		invariant subspace and sum zero, Duke Math. J. 118, no. 2 (2003), 339--352.
		\bibitem{DL} K. Diarra, and F. Loray, Normal forms for rank two linear irregular differential equations and moduli spaces, Period. Math. Hungar., Periodica Mathematica Hungarica.
		Journal of the J\'anos Bolyai Mathematical Society, 84, 2022, 2, 303--320.
		\bibitem{DST}A. Dzhamay, H. Sakai, and T. Takenawa, Discrete Schlesinger transformations, their Hamiltonian formulation, and Difference Painlev\'e equations, arXiv:1302.2972v2 [math-ph], 2013, pp. 1--29.
		\bibitem{DT}A. Dzhamay and T. Takenawa, Geometric analysis of reductions from Schlesinger transformations to difference Painlev\'e equations, Algebraic and analytic aspects of integrable systems and Painlev\'e equations, Contemp. Math., 651, 87--124, Amer. Math. Soc., Providence, RI,2015.
		\bibitem{Hi} K. Hiroe, Linear differential equations on the Riemann sphere and representations of
		quivers, Duke Math. J. 166 (2017), no. 5, 855--935.
		\bibitem{HF}Y. Haraoka and G.~V. Filipuk, 
		Middle convolution and deformation for Fuchsian systems,
		J. Lond. Math. Soc. (2) {\bf 76} (2007), no.~2, 438--450.
		\bibitem{HY} K. Hiroe, D. Yamakawa, Moduli spaces of meromorphic connection and quiver varieties,
		Adv. Math., 266 (2014), 120--151.
		\bibitem{IIS1} M. Inaba and K. Iwasaki and M. -H. Saito, Moduli of stable parabolic connections, {R}iemann-{H}ilbert
		correspondence and geometry of {P}ainlev\'{e} equation of type
		{VI}. {I},  Publ. Res. Inst. Math. Sci. 42 (2006), no. 4, 987--1089. 
		\bibitem{IIS2}  M. Inaba and K. Iwasaki and M. -H. Saito, Moduli of stable parabolic connections, {R}iemann-{H}ilbert correspondence and geometry of {P}ainlev\'{e} equation of type {VI}. {II},
		\bibitem{In} M. Inaba, Moduli of parabolic connections on curves and the
		{R}iemann-{H}ilbert correspondence, J. Algebraic Geom. 22 (2013), no. 3, 407--480. 
		\bibitem{IS1} M. Inaba, M.-H. Saito, Moduli of unramified irregular singular parabolic connections on a smooth projective curve. Kyoto J. Math. 53 (2013) 433--482.
		\bibitem{IS2} M. Inaba and M. -H. Saito, Moduli of regular singular parabolic connections with given spectral type on smooth projective curves, J. Math. Soc. Japan 70, No. 3 (2018), 879--894.
		\bibitem{KS} A. Komyo and M. -H. Saito, Explicit description of jumping phenomenon moduli
		spaces of parabolic connections and Hilbert schemes of points on surfaces. Kyoto J.
		Math. 59(3), 515--552 (2019)
		\bibitem{KLSS} A. Komyo, F. Loray, M. -H. Saito, and S. Szab\'o, canonical coordinates for moduli spaces of rank two irregular connections on curves, Preprint arXiv:2309.05012 (2023).
		\bibitem{KNY}K. Kajiwara, M. Noumi, and Y. Yamada. Geometric aspects of Painlev\'e equations. J. Phys. A, 50(7):073001, 164, 2017.
		\bibitem{Ma1} T. Matsumoto, Moduli space of rank three logarithmic connections on the projective line with three poles, Preprint arXiv:2311.10071 (2023).
		\bibitem{Ma2} T. Matsumoto, Realization of symmetry of $A^{(1)*}_2$-surfaces as transformations of logarithmic connections, Preprint arXiv:2508.07920 (2025).
		\bibitem{Sa1}H. Sakai, Rational surfaces associated with affine root systems and geometry of the Painlev\'e equations, Comm. Math. Phys., 220(1):165--229, 2001.
		\bibitem{Sim}C.~T. Simpson, Katz's middle convolution algorithm,
		Pure Appl. Math. Q. {\bf 5} (2009), no.~2, Special Issue: In honor of Friedrich Hirzebruch. Part 1, 781--852.
		\bibitem{SS}M. -H. Saito and S. Szab\'{o}, On apparent singularities and canonical coordinates for moduli spaces of connections, in preparation.
		\bibitem{Ya}D. Yamakawa. 
		Fourier-Laplace transform and isomonodromic deformations,
		Funkcial. Ekvac. 59 (2016), 315--349.
	\end{thebibliography}
\end{document}